\documentclass{amsart}

\usepackage{amsmath, amsfonts, amssymb,centernot, tikz, cite, amsthm}
\usepackage{enumerate, url,}
\usepackage[margin=1.6in]{geometry}

\newtheorem{thm}{Theorem}
\newtheorem{prop}[thm]{Proposition}
\newtheorem{lemma}[thm]{Lemma}
\newtheorem{mydef}[thm]{Definition}
\newtheorem{cor}[thm]{Corollary}
\newenvironment{remark}[1][Remark]{\begin{trivlist}
\item[\hskip \labelsep {\bfseries #1}]}{\end{trivlist}}

\DeclareMathOperator{\Tr}{Tr}
\DeclareMathOperator{\deter}{det}

\DeclareMathOperator{\im}{Im}
\DeclareMathOperator{\disc}{disc}

\begin{document}

\title[Perturbed periodic Jacobi matrices]{Spectral results for perturbed periodic Jacobi matrices using the discrete Levinson technique}

\author[E. Judge]{Edmund Judge}
\address{School of Mathematics, Statistics, and Actuarial Science\\ Sibson Building\\ University of Kent\\ Canterbury\\ Kent\\ CT2 7FS\\ United Kingdom}
\email{ej75@kent.ac.uk}

\author[S. Naboko]{Sergey Naboko}
\address{ Depart. of Math. Physics\\ Institute of Physics\\St. Petersburg State University\\
1 Ulianovskaia\\ St. Petergoff\\
St. Petersburg\\ 198504\\ Russia} \email{sergey.naboko@gmail.com}

\author[I. Wood]{Ian Wood}
\address{School of Mathematics, Statistics, and Actuarial Science\\ Sibson Building\\ University of Kent\\ Canterbury\\ Kent\\ CT2 7FS\\ United Kingdom}
\email{i.wood@kent.ac.uk}
\date{}

\begin{abstract}
 For an arbitrary Hermitian period-$T$ Jacobi operator, we assume a perturbation by a Wigner-von Neumann type potential to devise subordinate solutions to the formal spectral equation for a (possibly infinite) real set, $S$, of the spectral parameter. We employ discrete Levinson type techniques to achieve this, which allow the analysis of the asymptotic behaviour of the solution. This enables us to construct infinitely many spectral singularities on the absolutely continuous spectrum of the periodic Jacobi operator, which are stable with respect to an $l^1$-perturbation. An analogue of the quantisation conditions from the continuous case appears, relating the frequency of the oscillation of the potential to the quasi-momentum associated with the purely periodic operator.
\end{abstract}



\maketitle

\section{Introduction}

The simplest classical Levinson result roughly states that for the initial value problem $$y''+(\lambda^2-P(x))y=0, x\in\mathbb{R^+}, y(0,\lambda)=0,y'(0,\lambda)=1,$$ the solution will approach a sine function as $x$ tends to infinity, providing $\lambda$ is real and $P(x)$ obeys certain conditions, such as nonnegativity and tending to zero fast enough \cite{2,1}. In 1987, Benzaid and Lutz \cite{3} adapted and applied the theory to the study of discrete systems of the form \begin{equation}\label{7.1} x(n+1)=A(n+1)x(n), n\geq n_0\end{equation} where $x=(x(n))_{n\geq n_0}$ is a sequence of $\mathbb{C}^d$ vectors and $A=(A(n))_{n\geq n_0}$, a sequence of $d\times d$ complex matrices. Since then the assumptions on $A$ have been varied and investigated to better determine the effects on the spectrum of Jacobi matrices (see, for example, \cite{5, 6, 7,8,11, 12,25,mosz1,mosz2}). In this paper we explore Levinson type techniques and use a diagonal perturbation to produce subordinate solutions of the formal spectral equation (see~\eqref{4.25}) and to embed singularities into the absolutely continuous spectrum (a.c. spectrum) of periodic Hermitian Jacobi matrices, $J_T$. These are operators acting in $l^2(\mathbb{N})$ which are tri-diagonal and have the form \begin{equation}\label{4.10}
J_T:=\left(\begin{array}{ccccccccccccc}
 b_1&a_1&\\
 a_1&b_2&a_2&\\
 &a_2&b_3&a_3&\\
 &&\ddots&\ddots&\ddots&\\
 &&&a_{T-1}&b_T&a_T\\
 &&&&a_T&b_1&a_1&\\
 &&&&&a_1&b_2&a_2&\\
 &&&&&&\ddots&\ddots&\ddots&\\
 &&&&&&&a_{T-1}&b_{T}&a_T\\
 &&&&&&&&a_T&b_1&a_1&\\
 &&&&&&&&&a_1&b_2&a_{2}&\\
 &&&&&&&&&&\ddots&\ddots&\ddots
 \end{array}\right).\end{equation} We assume throughout that $a_i,b_i\in\mathbb{R}, a_i>0$ for all $i$.

 For Jacobi operators, $J$ (not necessarily periodic), we can employ the instrument of transfer matrices, $B_i(\lambda)$. The transfer
 matrices are given by \begin{equation}\label{4.2} B_{i}(\lambda):=\left(\begin{array}{cc}
0&1\\
\frac{-a_{i-1}}{a_i}&\frac{\lambda-b_i}{a_i}\end{array}\right), \lambda\in\mathbb{C}.\end{equation} They produce solutions (generalised eigenvectors or orthogonal polynomials) to the series of recurrence relations which determine the formal spectral equation, with parameter $\lambda$,
\begin{equation}
a_{n-1}u_{n-1}+b_nu_n+a_nu_{n+1}=\lambda u_n, n\geq 2,\label{4.25}
\end{equation}
 that must be satisfied to solve $J\underline{u}=\lambda \underline{u}$, where $\underline{u}:=(u_n)_{n\geq1}$. Sometimes we add the first row condition
  \begin{equation}b_1u_1+a_1u_2=\lambda u_1\label{4.28}\end{equation} to produce a particular generalized eigenvector (orthogonal polynomial in $\lambda$ of the first kind). In particular, the transfer matrices give the next component in the solution when two are already known, i.e. \begin{equation}\left(\begin{array}{c}
 u_n\\
 u_{n+1}\end{array}\right)=B_n(\lambda)\left(\begin{array}{c}
 u_{n-1}\\
 u_{n}\end{array}\right).\label{4.59}\end{equation} However, it should be stressed that this method using transfer matrices does not guarantee a sequence generated by \eqref{4.59} is an eigenvector for two reasons. Firstly, the transfer matrices might not necessarily encode the initial conditions,~\eqref{4.28}, of the Jacobi operator, and secondly if a subordinate (in our case decaying) solution should exist it is not guaranteed to belong to $l^2(\mathbb{N};\mathbb{C})$. This is where the need for more sophisticated asymptotic results, like Levinson's, arises.

A solution, $\underline{u}=(u_n)$, to \eqref{4.25} is said to be subordinate if and only if $$\lim_{N\rightarrow\infty} \frac{\|{\underline{u}}\|_N}{\|{\underline{v}}\|_N}=0,~~{\rm{where}}~~\|\underline{x}\|_N=\sqrt{\sum\limits_{n=1}^N |x_n|^2},$$ for any solution ${\underline{v}}=(v_n)_{n\geq 1}$ of~\eqref{4.25} not a constant multiple of ${\underline{u}}$. According to Gilbert-Pearson theory \cite{19,20} a detailed description of the spectral structure for periodic Jacobi operators can be inferred from the existence or non-existence of subordinate solutions to the formal spectral equation.

 Now let $M$ be the monodromy matrix for an arbitrary period-$T$ operator, i.e. \begin{equation}\label{4.1} M(\lambda):=B_{T}(\lambda)B_{T-1}(\lambda)\dots B_{1}(\lambda),\end{equation} with $a_0:=a_T$. Note $\deter(M(\lambda))\equiv 1$ and therefore the eigenvalues of $M(\lambda)$, $$\mu_\pm(\lambda):=\frac{\Tr M(\lambda)}{2}\pm \sqrt{\left(\frac{\Tr M(\lambda)}{2}\right)^2-1},$$ are algebraic functions in $\lambda$. We can use the matrix $M(\lambda)$ to canonically partition the points in the complex plane into three categories: hyperbolic, elliptic, parabolic.

\begin{mydef}
The hyperbolic points are those $\lambda\in\mathbb{C}$ that produce a monodromy matrix with two eigenvalues, $\mu_1,\mu_2$ such that
$|\mu_1|>1$ and $|\mu_2|<1$; elliptic points those that produce two distinct eigenvalues of modulus one; and parabolic points those that
produce one eigenvalue of  modulus one, equal to $1$ or $-1$, with algebraic multiplicity two.
\end{mydef}

\begin{remark}
One can easily see that the elliptic and parabolic points are real (see, for example, Lemma~2.5 in \cite{21}).
\end{remark}

For $\lambda\in\mathbb{R}$ we can distinguish the hyperbolic, elliptic and parabolic cases by $|\Tr(M(\lambda))|>2, |\Tr(M(\lambda))|< 2,
|\Tr(M(\lambda))|=2$, respectively. From the theory of Gilbert-Pearson it follows that elliptic and parabolic points lie in the a.c. spectrum, and as $\lambda\mapsto \Tr M(\lambda)$ is continuous, the a.c. spectrum for a period-$T$ Jacobi operator appears in bands: \begin{equation}\label{4.65}\sigma_{a.c.}(J_T)=\{\lambda\in\mathbb{R}~|~\Tr(M(\lambda))\leq 2\}.\end{equation}

\begin{mydef}
We define the generalised interior of the essential spectrum of $J_T$, denoted $\sigma_{ell}(J_T)$, to be the set of elliptic points.
\end{mydef}
 The first main theorem of the paper (recorded below) states that for any individual $\lambda$ in the generalised interior of the essential spectrum of an arbitrary period-$T$ Jacobi operator a potential, $(q_n)$, can be contrived such that the new diagonally-perturbed Jacobi operator has a subordinate solution at $\lambda$. In our case this is a solution, $\underline{u}:=(u_n)_{n\geq 1},$ to
\begin{equation}
a_{n-1}u_{n-1}+(b_n+q_n)u_n+a_nu_{n+1}=\lambda u_n, n\geq 2\label{4.26}
\end{equation}
that decays. Developing this idea further we construct a potential that is the sum of Wigner-von Neumann type potentials to produce infinitely many values of the spectral parameter (spectral singularities on the a.c. spectrum) of the period-$T$ Jacobi operator where subordinate solutions exist (see Theorem~\ref{transfer3}).

  Our focus in this paper is on subordinate solutions, and these are, in fact, stable with respect to an $l^1$-perturbation whereas eigenvalues are unstable with respect to arbitrarily small rank-$1$ perturbations (see, for example, \cite{sarg}). Although it is not proved here, we expect that the spectral density (i.e. the derivative of the spectral function) vanishes with power-like decay at points where we have subordinate solutions, giving pseudogaps in the spectrum (see \cite{28} and \cite{sim2} for the continuous case, and \cite{29} for the discrete Schr\"{o}dinger operator case). The potentials involved here to obtain these subordinate solutions have a Wigner-von Neumann structure (see, for example, \cite{24,26,27,28,32,30, 33}) which in its original conception for Schr\"{o}dinger equations had the form $$\frac{c\sin(2\omega
x+\varphi)}{x}.$$ Here it is adapted to the discrete setting accordingly, and is assumed a priori; the necessary size of the constant $c$ (or its analogue in our case) for the solution to lie in $l^2$, a subject of our investigation. Note that the first row condition, $$(q_1+b_1)u_1+a_1u_2=\lambda u_1$$ that either inhibits or enables the subordinate solution to become an eigenvector (and consequently, $\lambda$, an embedded eigenvalue) is dealt with later in the paper. The generalisation from the $1$-periodic case to the general periodic setting introduces new features, such as the set of points that are the roots to an algebraic function where the method fails.

\begin{thm}\label{transfer2}
For $\lambda\in\sigma_{ell}(J_T)$, let $e^{\pm i\theta(\lambda)}$ be the eigenvalues of $M(\lambda)$, where $\theta(\lambda)$ is the quasi-momentum. For any $\lambda\in\sigma_{ell}(J_T)$ outside an explicitly described finite set, we can choose $\omega$ s.t. $\omega T+2\theta(\lambda)\in2\pi\mathbb{Z}$ or $\omega T - 2\theta(\lambda)\in2\pi\mathbb{Z}$, and \begin{equation}\label{4.29} q_n=\frac{{c} \sin(n\omega+\phi)}{n}\end{equation} for some ${c}\in\mathbb{R}\setminus\{0\},\phi\in\mathbb{R}$, such that there exists a subordinate solution $\underline{u}:=\left(u_n\right)_{n\geq 1}$ to Equation~\eqref{4.26}. In this case, there exists a $\delta>0$ s.t. for $|{c}|>\delta$ the subordinate solution resides in $l^2$.
\end{thm}

\begin{remark}
We stress that the values of $\lambda\in\sigma_{ell}(J_T)$ for which the theorem holds are defined explicitly using the functions $E(\lambda), \widetilde{E}(\lambda),\widetilde{\widetilde{E}}(\lambda)$, given later in \eqref{4.31},\eqref{4.32} and \eqref{4.33}, respectively.
\end{remark}

The proof of the result is separated into five steps (Sections~\ref{transfersec1} to \ref{transfersec5}). Once the result regarding the subordinate solution for a single candidate eigenvector has been expounded, the initial conditions are discussed (Section \ref{transfersec5}) so that the value $\lambda$ becomes a formal eigenvalue. Then, in Section~\ref{transfersec6}, the technique is adapted to construct a collection of subordinate solutions corresponding to (possibly infinitely many) values of the spectral parameter in the elliptic spectrum. Finally we give an illustrative example for the special case of two candidate eigenvalues, where some extra conditions must be satisfied in order for the two subordinate solutions to become eigenvectors (Section~\ref{transfersec6}).

\begin{remark}
We direct the interested reader to our previous paper~\cite{21} in which we use a generalisation of the Wigner-von Neumann technique \cite{24} to embed a single eigenvalue into the essential spectrum of an arbitrary period-$T$ Jacobi operator. There, rather than assume a Wigner-von Neumann potential a priori, as we do in our present method, we deduce it from an ansatz we make for the eigenvector.
\end{remark}

\section{Variation of parameters}\label{transfersec1}
In this section we adopt a suitable change of discrete variables with the aim of simplifying the analysis of the transfer matrix product.

First, observe that for $n=kT$, for any solution $(u_n)$ to \eqref{4.26}, we have
 \begin{equation}\label{1}
 \left(\begin{array}{c}
 u_{n+T}\\
 u_{n+T+1}\end{array}\right)=:\vec{u}_{n+T}=M_{k}(\lambda)\vec{u}_n
 \end{equation}
where $$M_{k}(\lambda):=\left(\begin{smallmatrix}
0&1\\
-\frac{a_{T-1}}{a_T}&\frac{\lambda-b_T-q_{n+T}}{a_T}
\end{smallmatrix}\right)\left(\begin{smallmatrix}
0&1\\
-\frac{a_{T-2}}{a_{T-1}}&\frac{\lambda-b_2-q_{n+T-1}}{a_{T-1}}
\end{smallmatrix}\right)\dotsc\left(\begin{smallmatrix}
0&1\\
-\frac{a_T}{a_1}&\frac{\lambda-b_{1}-q_{n+1}}{a_1}
\end{smallmatrix}\right),$$ the perturbed monodromy matrix. Now considering only $n$ such that $n=kT$, define a new parameter $\vec{f}_k$ such that \begin{equation}\label{2}
\vec{u}_n=M^k(\lambda)\vec{f}_k,
\end{equation} where $M(\lambda)$ is the unperturbed monodromy matrix \eqref{4.1}. Substituting \eqref{2} into \eqref{1} gives
\begin{equation}\label{3.1}
\vec{u}_{n+T}=M_{k}(\lambda)M^k(\lambda)\vec{f}_k.
\end{equation} We define $$\Sigma_{k}(\lambda):=\sum\limits_{j=0}^{T-1} B_T(\lambda)\dotsc B_{T-(j-1)}(\lambda)\left(\begin{smallmatrix}
0&0\\
0&\frac{q_{T(k+1)-j}}{a_{T-j}}
\end{smallmatrix}\right)B_{T-(j+1)}(\lambda)\dots B_1(\lambda),$$ with $B_i(\lambda)$ as in \eqref{4.2}, and if the order of decreasing indices is formally violated we understand the corresponding product to be the identity matrix. By equating the information in Equations~\eqref{2} and {\eqref{3.1}}, and noting that $q_n=O\left(\frac{1}{n}\right)$, we obtain for $\lambda\in\sigma_{ell}(J_T)$
  \begin{align}
 \vec{f}_{k+1}&=M^{-k-1}(\lambda)M_{k}(\lambda)M^{k}(\lambda)\vec{f}_k\nonumber\\
 &=M^{-k-1}(\lambda)\left(M(\lambda)-\Sigma_{k}(\lambda)+O\left(\frac{1}{k^2}\right)\right)M^{k}(\lambda)\vec{f}_k\nonumber\\
 &=\left(I-M^{-k-1}(\lambda)\Sigma_{k}(\lambda) M^{k}(\lambda)+O\left(\frac{1}{k^2}\right)\right)\vec{f}_k\nonumber\\
 &=\bigg(I-V(\lambda)\left(\begin{smallmatrix}
\mu^{-k-1}&0\\
0&\overline{\mu}^{-k-1}\end{smallmatrix}\right)V^{-1}(\lambda)\Sigma_{k}(\lambda) V(\lambda)\left(\begin{smallmatrix}
\mu^{k}&0\\
0&\overline{\mu}^{k}\end{smallmatrix}\right)V^{-1}(\lambda)\nonumber\\
&+O\left(\frac{1}{k^2}\right)\bigg)\vec{f}_k,\label{4.22}
 \end{align}
 where we go from the second to the third line of the above calculation by using that $\|M^m(\lambda)\|$ is uniformly bounded for fixed $\lambda\in\sigma_{ell}(J_T)$ in $m\in\mathbb{Z}$; $\mu(\lambda)$ and $\overline{\mu(\lambda)}$ are the (conjugate) eigenvalues (of modulus 1) of the unperturbed transfer matrix, $M(\lambda)$, with $\lambda\in\sigma_{ell}(J_T); V(\lambda),V^{-1}(\lambda)$ are the matrices that diagonalise $M(\lambda)$.

\begin{remark}
It is sufficient to consider only $n$ of the type $n=kT$ for the general asymptotic analysis of our solution. This follows from the fact that for $n=kT+s$ where $s\in\{1,\dots,T-1\}$, we have the relation \begin{align}
\vec{u}_{kT+s}&=(B_{s}(\lambda-q_{kT+s})\dots B_{1}(\lambda-q_{kT+1}))M_{k-1}(\lambda)\dots M_{1}(\lambda)\vec{u}_0\nonumber\\
&=(B_{s}(\lambda-q_{kT+s})\dots B_{1}(\lambda-q_{kT+1}))\vec{u}_{Tk}\label{4.60}
\end{align}
and since the $B_j(\lambda-q_{kT+j})$ are invertible and $q_n$ tends to zero we have that for $k$ sufficiently large $$\|B_j(\lambda-q_{kT+j})\|\leq K_1; \|B_j^{-1}(\lambda-q_{kT+j})\|\leq K_2$$ for all $j\in\{1,\dots,T-1\}$, and finite $K_1,K_2$. Then \begin{align*}\|\vec{u}_{kT+s}\|&\leq K_1^s\|\vec{u}_{kT}\|, \|\vec{u}_{kT+s}\|\geq K_2^{-s}\|\vec{u}_{kT}\|,\\ \intertext{so} K_2^{-s}\|\vec{u}_{kT}\|&\leq \|\vec{u}_{kT+s}\|\leq K_1^{s}\|\vec{u}_{kT}\| \end{align*} and this with \eqref{4.60}, together with the fact $B_s(\lambda-q_{kT+s})\rightarrow B_s(\lambda)$ as $k\rightarrow\infty$, gives that the asymptotic behaviour of $\vec{u}_{kT}$ uniquely determines the asymptotic behaviour of $\vec{u}_{kT+s}$ for any $s\in\{0,\dots,T-1\}$. In other words, we can interpolate the asymptotic behaviour for $n=kT$ to arbitrary values of $n$.
\end{remark}
\section{Preparation for the Harris-Lutz procedure}\label{transfersec2}
In this section we apply variation of parameters again (this time on $\vec{f}_k$) and continue simplifying the expression down into something to which we can apply the Harris-Lutz procedure~\cite{13,15,25}. The Harris-Lutz procedure gives us a way to remove all the terms that do not affect the asymptotics from our analysis.

Observe that the unperturbed monodromy matrix, $M(\lambda)$, has the form $$M(\lambda)=\left(\begin{array}{cc}
p_1(\lambda)&p_2(\lambda)\\
p_3(\lambda)&p_4(\lambda)\end{array}\right),$$ where $p_1(\lambda), p_2(\lambda), p_3(\lambda), p_4(\lambda)$ are real polynomials in $\lambda$. For $j\in\{0,\dots,T-1\}$, $$B_{T}(\lambda)\dots B_{T-(j-1)}(\lambda)=\left(\begin{smallmatrix}
\alpha_1^{(j)}(\lambda)&\alpha_2^{(j)}(\lambda)\\
\alpha_3^{(j)}(\lambda)&\alpha_4^{(j)}(\lambda)\end{smallmatrix}\right)$$ and $$B_{T-(j+1)}(\lambda)\dots B_1(\lambda)=\left(\begin{smallmatrix}
\tilde{\alpha}_1^{(j)}(\lambda)&\tilde{\alpha}_2^{(j)}(\lambda)\\
\tilde{\alpha}_3^{(j)}(\lambda)&\tilde{\alpha}_4^{(j)}(\lambda)\end{smallmatrix}\right),$$where $\alpha_i^{(j)}(\lambda),\widetilde{\alpha}_i^{(j)}(\lambda)$ are also real polynomials in $\lambda$. Similarly to Lemma~2.1 from~\cite{21} we obtain the following result on the form of $\alpha_i^{(j)},\widetilde{\alpha}_i^{(j)}$ for $j\in\{0,\dots,T-1\}$:

\begin{lemma}\label{4.36}
For $B_i(\lambda), i\in\{1,\dots,T\}$, as described in Equation~\eqref{4.2}, we have that for $j\in\{0,\dots, T-1\}$
\begin{align*}
 \alpha_{1}^{(j)}(\lambda)&=\begin{cases} 1,~ j=0,\\ 0,~j=1,\\ -a_{T-j}\frac{\lambda^{j-2}}{\prod\limits_{s=T-j+1}^{T-1} a_s}+P_{j-3}(\lambda),~j\geq 2,\end{cases}\\
 \alpha_{2}^{(j)}(\lambda)&=\begin{cases} 0,~ j=0,\\ \frac{\lambda^{j-1}}{\prod\limits_{s=T-j+1}^{T-1} a_s}+P_{j-2}(\lambda),~j\geq 1,\end{cases}\\
 \alpha_{3}^{(j)}(\lambda)&=\begin{cases} 0,~ j=0,\\ -a_{T-j}\frac{\lambda^{j-1}}{\prod\limits_{s=T-j+1}^{T} a_s}+Q_{j-2}(\lambda),~j\geq 1,\end{cases}\\
 \alpha_{4}^{(j)}(\lambda)&=\frac{\lambda^j}{\prod\limits_{s=T-j+1}^T a_s}+P_{j-1}(\lambda),
 \end{align*}
 where $P_{j-1}(\lambda),P_{j-2}(\lambda),Q_{j-2}(\lambda)$ and $P_{j-3}(\lambda)$ are real polynomials in $\lambda$ of degree
 less than or equal to $j-1,j-2, j-2$ and $j-3$, respectively, and $P_k(\lambda)=0=Q_k(\lambda)$ for $k<0$.

 Similarly, for $j\in{0,\dots,T-1}$, we have
  \begin{align*}
 \widetilde{\alpha}_{1}^{(j)}(\lambda)&=\begin{cases} 1,~ j=T-1,\\ 0,~j=T-2,\\ -a_T\frac{\lambda^{T-j-3}}{\prod\limits_{s=1}^{T-j-2} a_s}+\widetilde{P}_{T-j-4}(\lambda),~j\leq T-3,\end{cases}\\
 \widetilde{\alpha}_{2}^{(j)}(\lambda)&=\begin{cases} 0,~ j=T-1,\\ \frac{\lambda^{T-j-2}}{\prod\limits_{s=1}^{T-j-2} a_s}+\widetilde{P}_{T-j-3}(\lambda),~j\leq T-2,\end{cases}\\
 \widetilde{\alpha}_{3}^{(j)}(\lambda)&=\begin{cases} 0,~ j=T-1,\\ -a_T\frac{\lambda^{T-j-2}}{\prod\limits_{s=1}^{T-j-1} a_s}+\widetilde{Q}_{T-j-3}(\lambda),~j\leq T-2,\end{cases}\\
 \widetilde{\alpha}_{4}^{(j)}(\lambda)&=\frac{\lambda^{T-j-1}}{\prod\limits_{s=1}^{T-j-1} a_s}+\widetilde{P}_{T-j-2}(\lambda),
 \end{align*}
 where $\widetilde{P}_{T-j-2}(\lambda),\widetilde{P}_{T-j-3}(\lambda),\widetilde{Q}_{T-j-3}(\lambda)$ and $\widetilde{P}_{T-j-4}(\lambda)$ are real polynomials in $\lambda$ of degree
 less than or equal to $T-j-2,T-j-3, T-j-3$ and $T-j-4$, respectively, and $\widetilde{P}_k(\lambda)=0=\widetilde{Q}_k$ for $k<0$.
 \end{lemma}

Now using the (small) freedom we have in diagonalising matrices $V(\lambda)$, for $M(\lambda), \lambda\in\sigma_{ell}(J_T)$, we can construct them such that the entries in the second row are equal to $1$. To show that this is always possible assume for contradiction that the eigenvector, $\vec{v}_1$, has second component zero, i.e. $\vec{v}_1=\left(\begin{array}{cc}
a\\
0\end{array}\right), a\neq 0$. Then $$\left(\begin{array}{cc}
p_1(\lambda)&p_2(\lambda)\\
p_3(\lambda)&p_4(\lambda)\end{array}\right)\left(\begin{array}{cc}
a\\
0\end{array}\right)=\mu\left(\begin{array}{cc}
a\\
0\end{array}\right).$$ Thus $$ap_1(\lambda)=\mu a\Rightarrow a(p_1(\lambda)-\mu)=0,$$
which gives $p_1(\lambda)=\mu$. However, $p_1(\lambda)\neq\mu$ because $p_1(\lambda)$ is real and, since $\lambda\in\sigma_{ell}(J_T)$, $\mu$ is non-real.  A simple calculation shows that $V(\lambda)=\left(\begin{array}{cc}
\frac{p_2(\lambda)}{\mu-p_1(\lambda)}&\frac{p_2(\lambda)}{\overline{\mu}-p_1(\lambda)}\\
1&1\end{array}\right)$. Note that the same reasoning as above shows that there are no eigenvectors of $M(\lambda)$, $\lambda\in\sigma_{ell}(J_T)$, with zero first component. Therefore, $V(\lambda)$ is always invertible for $\lambda\in\sigma_{ell}(J_T)$, as $\mu$ is non-real and $p_2(\lambda)\neq 0$. (The latter follows from the fact that if $p_2(\lambda)=0$ then the monodromy matrix is lower triangular and therefore the eigenvalues are the diagonal entries, which in this case are real.)

Define a new parameter, $\vec{g}_k$, such that \begin{equation}\label{4.51}\vec{g}_k:=V^{-1}(\lambda)\vec{f}_k.\end{equation} In terms of $\vec{g}_k$, Equation~\eqref{4.22} becomes
\begin{align}
\vec{g}_{k+1}&=\left(I-\left(\begin{array}{cc}
\overline{\mu}^{k+1}&0\\
0&{\mu}^{k+1}\end{array}\right)V^{-1}(\lambda)\Sigma_{k+1}(\lambda) V(\lambda)\left(\begin{array}{cc}
\mu^{k}&0\\
0&\overline{\mu}^{k}\end{array}\right)+O\left(\frac{1}{k^2}\right)\right)\vec{g}_k\nonumber\\
&=\bigg(I-\sum\limits_{j=0}^{T-1}\frac{q_{(k+1)T-j}}{a_{T-j}}\left(\begin{array}{cc}\overline{\mu}^{k+1}&0\\
0&{\mu}^{k+1}\end{array}\right)V^{-1}(\lambda)B_{T}(\lambda)\dots B_{T-(j-1)}(\lambda)\nonumber\\
&~~~\times\left(\begin{array}{cc}
0&0\\
0&1\end{array}\right)B_{T-(j+1)}(\lambda)\dots B_1(\lambda)V(\lambda)\left(\begin{array}{cc}
\mu^{k}&0\\
0&\overline{\mu}^{k}\end{array}\right)+O\left(\frac{1}{k^2}\right)\bigg)\vec{g}_k\nonumber\\
&=\bigg(I-\sum\limits_{j=0}^{T-1}\frac{q_{(k+1)T-j}}{a_{T-j}}\left(\begin{array}{cc}\overline{\mu}^{k+1}&0\\
0&{\mu}^{k+1}\end{array}\right)\left(\begin{array}{cc}
\frac{p_2}{\mu-p_1}&\frac{p_2}{\overline{\mu}-p_1}\\
1&1\end{array}\right)^{-1}\left(\begin{array}{cc}
\alpha_1^{(j)}&\alpha_2^{(j)}\\
\alpha_3^{(j)}&\alpha_4^{(j)}\end{array}\right)\nonumber\\
&~~\times\left(\begin{array}{cc}
0&0\\
0&1\end{array}\right)\left(\begin{array}{cc}
\tilde{\alpha}_1^{(j)}&\tilde{\alpha}_2^{(j)}\\
\tilde{\alpha}_3^{(j)}&\tilde{\alpha}_4^{(j)}\end{array}\right)\left(\begin{array}{cc}
\frac{p_2}{\mu-p_1}&\frac{p_2}{\overline{\mu}-p_1}\\
1&1\end{array}\right)\left(\begin{array}{cc}
\mu^{k}&0\\
0&\overline{\mu}^{k}\end{array}\right)+O\left(\frac{1}{k^2}\right)\bigg)\vec{g}_k\label{4.23}\end{align}
Consequently,
\begin{multline}
\vec{g}_{k+1}=\bigg(I+\frac{1}{i\sin\theta(\lambda)}\sum\limits_{j=0}^{T-1}\frac{q_{(k+1)T-j}}{a_{T-j}}\bigg\{\left(\begin{smallmatrix}
-\overline{C_j(\lambda)}&0\\
0&C_j(\lambda)
\end{smallmatrix}\right)\\+\left(\begin{smallmatrix}
0& -\overline{D_j(\lambda)}\overline{\mu}^{2k}\\
D_j(\lambda)\mu^{2k}&0
\end{smallmatrix}\right)\bigg\}+O\left(\frac{1}{k^2}\right)\bigg)\vec{g}_k,
\end{multline} where the explicit calculation of the product of the seven matrices in \eqref{4.23} gives \begin{multline}\label{4.74}C_j(\lambda):=\frac{|\mu-p_1(\lambda)|^2\mu}{2p_2(\lambda)}\left(\tilde{\alpha}_3^{(j)}(\lambda)\frac{p_2(\lambda)}{\overline{\mu}-p_1(\lambda)}+\tilde{\alpha}_4^{(j)}(\lambda)\right)\\
\times\left(-\alpha_2^{(j)}(\lambda)+\alpha_4^{(j)}(\lambda)\frac{p_2(\lambda)}{\mu-p_1(\lambda)}\right)\end{multline} and \begin{multline}\label{4.66}D_j(\lambda):=\frac{|\mu-p_1(\lambda)|^2\mu}{2p_2(\lambda)}\left(\tilde{\alpha}_3^{(j)}(\lambda)\frac{p_2(\lambda)}{\mu-p_1(\lambda)}+\tilde{\alpha}_4^{(j)}(\lambda)\right)\\
\times\left(-\alpha_2^{(j)}(\lambda)+\alpha_4^{(j)}(\lambda)\frac{p_2(\lambda)}{\mu-p_1(\lambda)}\right).\end{multline}

\begin{remark}
Observe that $C_j(\lambda)\neq 0, D_j(\lambda)\neq 0$ for $\lambda\in\sigma_{ell}(J_T)$. Indeed, we prove this property only for $C_j(\lambda)$ (the argument is similar for $D_j(\lambda)$). As $\mu$ is non-real, the first two sets of brackets in the definition of $C_j(\lambda)$ are non-zero. Now considering only the third set of brackets we see that if $\widetilde{\alpha}_3^{(j)}(\lambda)$ is non-zero then the bracket is non-vanishing, since $\overline{\mu}$ is non-real. Otherwise, for this bracket to vanish means $\widetilde{\alpha}_4^{(j)}(\lambda)$ should also be zero, and if so the determinant of the matrix product $B_{T}(\lambda)\dots B_{T-{j-1}}(\lambda)$ is zero, but this never happens. That the fourth, and final, set of brackets is non-zero follows similarly.
\end{remark}

\section{Application of the Harris-Lutz procedure}\label{transfersec3}
Here we employ the Harris-Lutz procedure which will permit the removal of the matrix with components $C_j$ from the expression defining $\vec{g}_{k+1}$. As the $C_j$ term contains no oscillation it cannot cancel the oscillation from the potential and therefore this term can be eliminated using a suitable Harris-Lutz transformation. We will use the following proposition to explore the oscillation properties of the potential, $q_n$:
\begin{prop}\label{zyg}{\rm{(see~\cite{18})}}. Assume $\alpha,\gamma, \widetilde{c}\in\mathbb{R}$ and $\widetilde{c}\geq 0,\gamma>0$, then the following holds:
$$\sum\limits_{k=n}^\infty \frac{e^{ik\alpha}}{k^\gamma+\widetilde{c}}=O\left(\frac{1}{n^\gamma}\right),n\rightarrow\infty,~~~\iff
\frac{\alpha}{2\pi}\not\in\mathbb{Z}.$$
\end{prop}

Now recall that $$q_{n}=\frac{{c}\sin(n\omega+\phi)}{n}$$ for some ${c}\in\mathbb{R}\setminus\{ 0\}, \phi\in\mathbb{R}$. Clearly, $q_n=O\left(\frac{1}{n}\right)$. Moreover, by assumption either $\omega T +2\theta(\lambda)\in 2\pi\mathbb{Z}$ or $\omega T -2\theta(\lambda)\in 2\pi\mathbb{Z}$, which implies $\omega T\not\in 2\pi\mathbb{Z}$ (since $0<\theta(\lambda)<\pi)$.

The Harris-Lutz technique can now be employed to simplify the recurrence equation in \eqref{4.23}. First, define $\vec{h}_k$ such that \begin{equation}\label{4.46}\vec{g}_k=(I+G_k)\vec{h}_k\end{equation} for some $G_k=O\left(\frac{1}{k}\right) \in\mathbb{C}^{2\times 2}$ that has yet to be defined. Then Equation~\eqref{4.23} becomes
\begin{align}
\vec{h}_{k+1}&=(I+G_{k+1})^{-1}\bigg(I+\frac{1}{i\sin\theta(\lambda)}\sum\limits_{j=0}^{T-1} \frac{q_{(k+1)T-j}}{a_{T-j}}\bigg\{\left(\begin{array}{cc}
-\overline{C_j(\lambda)}&0\\
0&C_j(\lambda)\end{array}\right)\nonumber\\
&~~~~~~~~+\left(\begin{array}{cc}
0&-\overline{D_j(\lambda)}\overline{\mu}^{2k}\\
D_j(\lambda)\mu^{2k}&0\end{array}\right)\bigg\}+O\left(\frac{1}{k^2}\right)\bigg)(I+G_k)\vec{h}_k,\label{4.24}
\end{align}
and by Neumann series $$\left(I+G_{k+1}\right)^{-1}=I-G_{k+1}+O\left(\frac{1}{k^2}\right),$$ providing $\|G_{k+1}\|\leq \frac{1}{2}$, strictly less than one. Generally, this condition need not be true, however we may assume this without loss of generality. Indeed, for large values of $k$ the condition is true, and one can rearrange the formula for $G_k$ putting $G_k=0$, for $k=1,2,3,\dots, N$ for $N$ sufficiently large. It is clear this correction will serve the same goal for the Harris-Lutz transformation satisfying the smallness condition. In what follows we will use this idea every time we use the Harris-Lutz transformation without especially mentioning it.

Define the functions $$T_1(k):=\frac{1}{i\sin\theta(\lambda)}\sum\limits_{j=0}^{T-1} \frac{q_{(k+1)T-j}}{a_{T-j}}\left(\begin{array}{cc}
-\overline{C_j(\lambda)}&0\\
0&C_j(\lambda)\end{array}\right)$$ and $$T_2(k):=\frac{1}{i\sin\theta(\lambda)}\sum\limits_{j=0}^{T-1} \frac{q_{(k+1)T-j}}{a_{T-j}}\left(\begin{array}{cc}
0&-\overline{D_j(\lambda)}\overline{\mu}^{2k}\\
D_j(\lambda)\mu^{2k}&0\end{array}\right).$$  As $T_1(k)$ and $T_2(k)$ are of order $k^{-1}$, we see that \begin{align*}
\vec{h}_{k+1}&=(I-G_{k+1})(I+T_1(k)+T_2(k))(I+G_k)\vec{h}_k+O\left(\frac{1}{k^2}\right)\vec{h}_k\\
&=\left(I+G_k-G_{k+1}+T_1(k)+T_2(k)+O\left(\frac{1}{k^2}\right)\right)\vec{h}_k.
\end{align*} Furthermore, defining $G_k:=-\sum\limits_{l=0}^\infty T_1(k+l)$ and letting \newline $F_j:=\frac{1}{a_{T-j}}\left(\begin{array}{cc}
-\overline{C_j(\lambda)}&0\\
0&C_j(\lambda)\end{array}\right)$, $\kappa:=\frac{1}{i\sin\theta(\lambda)}$, gives $$T_1(k)=\kappa\sum\limits_{j=0}^{T-1}F_jq_{(k+1)T-j},$$ so \begin{align*}G_k&=-\kappa\sum\limits_{l=0}^\infty\sum\limits_{j=0}^{T-1}F_jq_{(k+1)T-j+lT}\\
&=-\kappa\sum\limits_{l=0}^\infty\sum\limits_{j=0}^{T-1}F_j\im \frac{ce^{i(((k+1)T-j+lT)\omega+\phi)}}{(k+1)T-j+lT}=O\left(\frac{1}{k}\right),\end{align*} using Proposition~\ref{zyg} and consequently $G_k$ is well-defined. Then,
\begin{align}
G_{k+1}-G_k=-\sum\limits_{l=1}^\infty T_1(k+l)+\sum\limits_{l=0}^\infty T_1(k+l)=T_1(k)=O\left(\frac{1}{k}\right)\label{4.48},
\end{align}
 and the Harris-Lutz procedure is successful meaning that Equation~\eqref{4.23} can now be written as \begin{equation}\label{4.35}\vec{h}_{k+1}=\left(I+T_2(k)+O\left(\frac{1}{k^2}\right)\right)\vec{h}_k.\end{equation}

In the next section, we will use the Harris-Lutz procedure to get rid of the $T_2(k)$ term for almost every value $\lambda\in\sigma_{ell}(J_T)$, specifically those that do not satisfy the so-called quantisation conditions.

\section{The necessity of quantisation conditions}\label{transfersec4}
Here the effects of the Harris-Lutz procedure applied previously are seen. Moreover, in its new form the recurrence equation for $\vec{h}_{k}$ can be rearranged, again, to clarify the role of the potential and the conditions for resonance seen; specifically, what values of $\theta(\lambda)$ prohibit another application of the Harris-Lutz procedure to the entire expression.

So, in the aftermath of the Harris-Lutz procedure, we have
\begin{align*}
\vec{h}_{k+1}&=\left(I+T_2(k)+O\left(\frac{1}{k^2}\right)\right)\vec{h}_k\\
&=\bigg(I+\frac{1}{i\sin\theta(\lambda)}\sum\limits_{j=0}^{T-1}\frac{{c}\sin({((k+1)T-j)\omega+\phi})}{a_{T-j}((k+1)T-j)}\left(\begin{smallmatrix}
0&-\overline{D_j(\lambda)}\overline{\mu}^{2k}\\
D_j(\lambda)\mu^{2k}&0\end{smallmatrix}\right)\\
&~~~+O\left(\frac{1}{k^2}\right)\bigg)\vec{h}_k\\
&=\bigg(I+\frac{1}{i\sin\theta(\lambda)}\sum\limits_{j=0}^{T-1}\frac{{c}e^{i(((k+1)T-j)\omega+\phi)}-{c}e^{-i(((k+1)T-j)\omega+\phi)}}{2ia_{T-j}((k+1)T-j)}\\
&~~~\times\left(\begin{smallmatrix}
0&-\overline{D_j(\lambda)}e^{-2ik\theta}\\
D_j(\lambda)e^{i2k\theta}&0\end{smallmatrix}\right)+O\left(\frac{1}{k^2}\right)\bigg)\vec{h}_k.
\end{align*}

Then, using the relation
\begin{align*}
\frac{1}{(k+1)T-j}=\frac{1}{kT}+O\left(\frac{1}{k^2}\right),
\end{align*} we obtain
\begin{align}
\vec{h}_{k+1}&=\bigg(I+\frac{{c}}{k\sin\theta(\lambda)}\sum\limits_{j=0}^{T-1}\bigg\{\left(\begin{smallmatrix}
0&-\overline{a_j(\lambda)}+\overline{b_j(\lambda)}\\
-a_j(\lambda)+b_j(\lambda)&0\end{smallmatrix}\right)+O\left(\frac{1}{k^2}\right)\bigg)\vec{h}_k,\label{4.67}
\end{align} where $$a_j(\lambda):=E_j(\lambda)e^{i(k(2\theta(\lambda)+\omega T)+(T-j)\omega+\phi)},$$ $$b_j(\lambda):=E_j(\lambda)e^{-i(k(\omega T-2\theta(\lambda))+(T-j)\omega+\phi)}$$ and \begin{equation}\label{4.68} E_j(\lambda):=\frac{D_j(\lambda)}{2Ta_{T-j}}.\end{equation}

It is natural to ask whether the Harris-Lutz technique can be applied, again, in order to further simplify the recurrence equation. The next result shows that this can be done whenever the so-called quantisation conditions \begin{equation}\label{4.30}\omega T\pm 2\theta(\lambda)\in2\pi\mathbb{Z}\end{equation} are not satisfied and that for $\lambda$ not satisfying the quantisation conditions, we only have oscillating solutions. However, decay is needed for a subordinate solution.

\begin{remark}
The quantisation formula gives the only possible location for eigenvalues in the a.c. spectrum involving integer parameters in the style of the Bohr-Sommerfield condition (see, for example, \cite{bohr}). For the (continuous) periodic Schr\"{o}dinger operator case this appears in~\cite{32}.
\end{remark}

\begin{thm}\label{4.44}
Assume $\lambda\in\sigma_{ell}(J_T)$ does not satisfy either of the quantisation conditions in \eqref{4.30}. Then there is no subordinate solution to the perturbed recurrence relations in \eqref{4.26}. Moreover, each non-zero solution of the relation is not increasing and purely oscillating, exactly like the solution to the unperturbed system described by \eqref{4.25}.
\end{thm}

\begin{proof}
We introduce a new sequence of vectors $\vec{l}_k$ such that \begin{equation}\label{4.62}\vec{h}_k=(I+H_k)\vec{l}_k\end{equation} and where $H_k$ will be defined below to satisfy $H_k=O\left(\frac{1}{k}\right)$. Then, \eqref{4.35} implies
\begin{align*}
\vec{l}_{k+1}&=\left(I+H_{k+1}\right)^{-1}\left(I+T_2(k)+O\left(\frac{1}{k^2}\right)\right)\left(I+H_k\right)\vec{l}_{k}\\
&=\left(I+H_{k}-H_{k+1}+T_2(k)+O\left(\frac{1}{k^2}\right)\right)\vec{l}_k.
\end{align*} By Proposition~\ref{zyg} we have $$\sum\limits_{m=k}^\infty \sum\limits_{j=0}^{T-1}\frac{e^{i(m(\omega T\pm2\theta(\lambda))+(T-j)\omega+\phi)}}{m}=O\left(\frac{1}{k}\right),$$ since both $\omega T\pm 2\theta(\lambda)\not\in 2\pi\mathbb{Z}$ and choosing $$H_k=-\sum\limits_{j=k}^\infty T_2(j)=O\left(\frac{1}{k}\right)$$ with $$H_{k+1}-H_{k}=T_2(k)$$
we obtain $$\vec{l}_{k+1}=\left(I+O\left(\frac{1}{k^2}\right)\right)\vec{l}_k.$$ Without loss of generality we assume the matrices $(I+O\left(\frac{1}{k^2}\right))$ are invertible for all $k\in\mathbb{N}$. Moreover, using an elementary result (for example Lemma 2.1 in \cite{7}) we have \begin{equation}\label{4.61}\vec{l}_k=(C+o\left(1\right))\vec{l}_1,\end{equation} where $C\in\mathbb{C}^{2\times2}$ is invertible and $\lim_{k\rightarrow\infty}\vec{l}_k=C\vec{l}_1$. Then, substituting \eqref{4.61} into \eqref{4.62} we obtain \begin{equation}\label{4.63} \vec{h}_k=(I+H_k)(C+o(1))\vec{l}_1 \end{equation} and substituting this into \eqref{4.46} gives \begin{equation}\label{4.64} \vec{g}_k=(I+G_k)(I+H_k)(C+o(1))\vec{l}_1.\end{equation} Substituting \eqref{4.64} into \eqref{4.51} gives \begin{equation} \vec{f}_k=V(\lambda)(I+G_k)(I+H_k)(C+o(1))\vec{l}_1\end{equation} and, in turn, substituting this into \eqref{3.1} we obtain \begin{equation} \vec{u}_{kT}=M^k(\lambda)V(\lambda)(I+G_k)(I+H_k)(C+o(1))\vec{l}_1.\end{equation} Finally, recalling that $H_k,G_k\rightarrow 0$ and $B_j(\lambda+q_{kT+j})\rightarrow B_j(\lambda)$ as $k\rightarrow\infty$  we have \begin{equation}\vec{u}_{kT+s}=B_s(\lambda)B_{s-1}(\lambda)\dots B_1(\lambda)M^k(\lambda)\left(\vec{r}(\lambda)+o\left(1\right)\right) \end{equation} for $s\in\{0,\dots,T-1\}$ and where $\vec{r}:=V(\lambda)C\vec{l}_1\in\mathbb{C}^2$ which is arbitrary since $\vec{l}_1$ is arbitrary. Consequently, the solution to the perturbed system, \eqref{4.26}, behaves like the solution to the unperturbed system, \eqref{4.25}. Moreover, the solutions are bounded from above and therefore there are no subordinate solutions~by the generalised Behnke-Stolz Lemma (see Lemma~2.2 in \cite{7}).~\qedhere
\end{proof}

\begin{remark} The set of $\lambda$ satisfying the quantisation conditions is discrete and since the intervals of a.c. spectrum are closed the
theorem shows that $$\sigma_{a.c.}(J_T)\subseteq\sigma_{a.c.}(J_T+Q).$$ Moreover, since the a.c. spectrum always belongs to the essential spectrum, and the Weyl-Theorem gives that the essential spectrum for the perturbed periodic Jacobi operator is the same as for the unperturbed periodic Jacobi operator, we also have $$\sigma_{a.c.}(J_T+Q)\subseteq\sigma_{ess}(J_T+Q)=\sigma_{ess}(J_T).$$ Finally, we use the fact that a periodic Jacobi operator only has finitely many eigenvalues (all isolated and of finite geometric multiplicity) and by \cite{20} the singular continuous spectrum is empty, so that $$\sigma_{ess}(J_T)=\sigma_{a.c.}(J_T).$$ All together these give $$\sigma_{a.c.}(J_T)=\sigma_{a.c.}(J_T+Q).$$
\end{remark}

\section{Resonance cases and asymptotic behaviour of subordinate solutions}\label{transfersec5}

In this section, the final steps of the method are carried out. Indeed, each of the quantisation conditions are considered, giving three resonance cases in total. In each of the resonance cases, various techniques are employed (including the Harris-Lutz transformation, again, although not to the entire expression which the resonance cases  prohibit) so that ultimately it is established that up to a few exceptions, regardless of what resonance case we are in, a decaying solution exists.

Without loss of generality, in the consideration below we confine ourselves to one band of $\sigma_{ell}(J_T)$. Choose $\omega$ such that $0<\omega<2\pi$. All the resonance cases can be described as follows:   $${\rm{Case}}~1:2\theta(\lambda)+\omega T=2k_+\pi,~{\rm{where}}~k_+\in\{1,\dots,T~\},~\omega T\not\in\pi\mathbb{Z}.$$ This range of $k_+$ is a consequence of $0<\omega T+2\theta(\lambda)<2\pi(T+1)$. $${\rm{Case}}~2: 2\theta(\lambda)-\omega T=-2k_-\pi,~{\rm{where}}~k_-\in\{0,\dots,T-1\},~\omega T\not\in\pi\mathbb{Z}.$$ The range of $k_-$ follows as a similar consideration to $k_+$. And, finally,  the special case where both first conditions in Cases 1 and 2 are satisfied. $${\rm{Case}}~3: 2\theta(\lambda)+\omega T=2k_+\pi, 2\theta(\lambda)-\omega T=-2k_-\pi,$${\rm{where}}~$k_+\in\{1,\dots,T\},k_-\in\{0,\dots,T-1\}$. The range of $k_+,k_-$ follow similarly to before. Indeed, by considering $\theta(\lambda)$ we see that here $k_-=k_+-1, \theta(\lambda)=\frac{\pi}{2}$ (which corresponds to the generalised `midpoint' of one band of $\sigma_{ell}(J_T)$) and $\omega T=(k_++k_-)\pi$. Note that according to Theorem~\ref{4.44} one of these three cases will need to hold to obtain a subordinate solution.

Furthermore, since we will be discussing the asymptotics of recurrences we introduce an equivalence on the set of recurrences, two recurrences being equivalent when their solutions have the same asymptotic behaviour. Specifically, we say $(a)\sim_r (b)$ where
\begin{align*}
(a):~~~~a_{k+1}&=A_ka_k~~\forall k\in\mathbb{N}\\
(b):~~~~b_{k+1}&=B_kb_k~~\forall k\in\mathbb{N} \end{align*} with $A_k,B_k\in\mathbb{C}^{2\times2}$ and invertible for all $k>N$, for some $N\in\mathbb{N}$, whenever for any solution $(a_n)$ of $(a)$ there exists a solution $(b_n)$ of $(b)$ such that $a_n=C_nb_n$, for all $n>N$ and where $\lim_{n\rightarrow\infty}C_n$ exists and is invertible. It is convenient to formally ignore possible non-invertibility of matrices $A_k$ and $B_k$ for small values of $k$ since we are interested only in the asymptotic behavior of the solutions.

Before we start discussing the separate cases we state and prove a lemma that will be used in the following arguments.

\begin{lemma}\label{4.34}
Let $A\in\mathbb{C}$ and $\widetilde{c}\in\mathbb{R}$. The recurrence \begin{equation}\label{4.69}(u):~~~~ \vec{u}_{k+1}=\left(I+\frac{\widetilde{c}}{k}\left(\begin{array}{cc}
0&-{\overline{A}}\\
-{A}&0\end{array}\right)+O\left(\frac{1}{k^2}\right)\right)\vec{u}_k~~\forall k\end{equation} is equivalent to the recurrence $$(v):~~~~  \vec{v}_{k+1}=\left(I+\frac{\widetilde{c}}{k}\left(\begin{array}{cc}
0&-{\overline{A}}\\
-{A}&0\end{array}\right)\right)\vec{v}_k~~\forall k,$$ i.e. $(u)\sim_r (v)$.
\end{lemma}

\begin{proof}
 The case $\widetilde{c}=0$ is trivial. For the case $\widetilde{c}\neq0$, the result follows from a generalisation of the Janas-Moszynski result (see Theorem 2 in~\cite{17}), however we must check that the following conditions are satisfied. We need to write \eqref{4.69} in the following form  $$\vec{u}_{k+1}=\left(I+p_kV_k+R_k\right)\vec{u}_k,$$ where \begin{enumerate} \item $p_k\geq 0, p_k\rightarrow 0$ and $\sum\limits_{k=1}^\infty p_k=\infty$, \item $\{R_k\}$ is a sequence of $2\times 2$ matrices each matrix element belonging to the sequence space $l^1$, \item $\{V_k\}$ such that $\sum\limits_{k=1}^\infty \|V_{k+1}-V_k\|<\infty$ with $\disc V_k>0$ and satisfying $\disc\left(\lim_{k\rightarrow\infty} V_k\right)\neq 0$, where $\disc V_k:=\left(\Tr(V_k)\right)^2-4\deter(V_k)$.\end{enumerate} Defining $p_k:=\frac{1}{k}$, the first condition is satisfied. Then defining $R_k$ as the error term of matrices of order $O\left(\frac{1}{k^2}\right)$ we see that the second condition is also satisfied. Finally, defining $$V_k:=\widetilde{c}\left(\begin{array}{cc}
0&-{\overline{A}}\\
-{A}&0\end{array}\right),$$ we see that $V_k$ is just a constant matrix sequence and immediately satisfies the first constraint in condition 3, the other two following from the conjugate entries of the matrix.~\qedhere
\end{proof}

 We now resume our discussion of the different cases.

{\bf{{Case 1}}}
Here, $\omega T+2\theta(\lambda)=2k_+\pi,~~\omega T\not\in\pi\mathbb{Z}$. We have from \eqref{4.67} that \begin{align*}
&\vec{h}_{k+1}=\bigg(I+\frac{{c}}{k\sin\theta(\lambda)}\sum\limits_{j=0}^{T-1}\bigg\{\left(\begin{array}{cc}
0&-\overline{E_j(\lambda)}e^{-i((T-j)\omega+\phi)}\\
-E_j(\lambda)e^{i((T-j)\omega+\phi)}&0\end{array}\right)\\
&+\left(\begin{array}{cc}
0&\overline{E_j(\lambda)}e^{i(2k\omega T+(T-j)\omega+\phi)}\\
E_j(\lambda)e^{-i(2k\omega T+(T-j)\omega+\phi)}&0\end{array}\right)\bigg\}+O\left(\frac{1}{k^2}\right)\bigg)\vec{h}_k.
\end{align*}

Then the Harris-Lutz procedure (i.e. a substitution of the form $\vec{h}_k=(I+\hat{H}_k)\vec{m}_k$) can be used again to get rid of the oscillating second term (as $\omega T\not\in\pi\mathbb{Z})$. Then, removing the error term using Lemma~\ref{4.34}, and for a suitable choice of $\hat{H}_k$ (similar to the proof of Theorem~\ref{4.44}), we have
\begin{align}
\vec{m}_{k+1}&=\bigg(I+\frac{{c}}{k\sin\theta(\lambda)}\sum\limits_{j=0}^{T-1}\left(\begin{smallmatrix}
0&-\overline{E_j(\lambda)}e^{-i((T-j)\omega+\phi)}\\
-E_j(\lambda)e^{i((T-j)\omega+\phi)}&0\end{smallmatrix}\right)\\
&~~+O\left(\frac{1}{k^2}\right)\bigg)\vec{m}_k\nonumber\\
\sim_r \vec{m}_{k+1}&=\bigg(I+\frac{{c}}{k\sin\theta(\lambda)}\sum\limits_{j=0}^{T-1}\left(\begin{smallmatrix}
0&-\overline{E_j(\lambda)}e^{-i((T-j)\omega+\phi)}\\
-E_j(\lambda)e^{i((T-j)\omega+\phi)}&0\end{smallmatrix}\right)\bigg)\vec{m}_k\nonumber\\
&=\bigg(I+\frac{{c}}{k\sin\theta(\lambda)}\left(\begin{array}{cc}
0&-{\overline{E(\lambda;k_+)}}\\
-{E(\lambda;k_+)}&0\end{array}\right)\bigg)\vec{m}_k,\label{4.52}
\end{align}
where \begin{equation}\label{4.31} E(\lambda;k_+):=\sum\limits_{j=0}^{T-1}E_j(\lambda)e^{i((T-j)\omega+\phi)}.\end{equation}

Observing that $\omega=\omega(\lambda;k_+)=\frac{-2\theta(\lambda)+2k_+\pi}{T}$ we see $$e^{i(T-j)\omega}=\mu^{-2}(\lambda)\left(\mu^{\frac{2}{T}}(\lambda)e^{-\frac{i2k_+\pi}{T}}\right)^{j}.$$ By Corollary~2.2 in \cite{21}, $$\Tr M(\lambda)=\mu(\lambda)+\frac{1}{\mu(\lambda)}\sim \frac{\lambda^T}{\prod\limits_{s=1}^T a_s}~{\rm as}~\lambda\rightarrow\infty.$$ We choose the branch of the square-root so that $\mu(\lambda)$ is decreasing as $\lambda\rightarrow\infty$ and thus it follows that $\mu(\lambda)\sim\frac{\prod\limits_{s=1}^T a_s}{\lambda^T}$. Also, there exists an appropriate branch of $\left(\mu^2(\lambda)\right)^{\frac{1}{T}}$ such that $$\mu^{\frac{2}{T}}(\lambda)=e^{\frac{i2\pi l_+}{T}}\left(\mu_g^2(\lambda)\right)^{\frac{1}{T}},$$ for some $l_+\in\{1,\dots, T\}$, with $\left(\mu_g^2(\lambda)\right)^{\frac{1}{T}}\sim\frac{\prod\limits_{s=1}^T a_s^{\frac{2}{T}}}{\lambda^2}$ as $\lambda\rightarrow\infty$. Then, \begin{align*}e^{i(T-j)\omega}&=\mu^{-2}(\lambda)\left(\left(\mu_g^2(\lambda)\right)^{\frac{1}{T}}e^{\frac{i2 \pi l_+}{T}}e^{-\frac{i2k_+\pi}{T}}\right)^{j}\\
&=\mu^{-2}(\lambda)\left(\mu_g^2(\lambda)\right)^{\frac{j}{T}},
\end{align*} if $k_+$ is chosen such that $k_+=l_+$. This particular choice of $k_+$, which will vary depending on $\lambda$, ensures that there is no oscillation occurring in the expression $e^{i(T-j)\omega}$ between different values of $j$. We denote \begin{equation}E(\lambda):=E(\lambda;l_+).\end{equation}

Then $E(\lambda)=e^{i\phi}\mu^{-2}(\lambda)\sum\limits_{j=0}^{T-1}E_j(\lambda)\left(\mu_g^2(\lambda)\right)^{\frac{j}{T}}.$
\begin{lemma}\label{4.38}
The function $E(\lambda)$ is algebraic and is not identically zero.
\end{lemma}

\begin{proof}
From the explicit formula \eqref{4.68} for $E_j(\lambda)$, and $\frac{1}{\mu(\lambda)}$, we see that $E(\lambda)$ is algebraic. To show that the function is not identically zero, we consider the case $T=1$ first. Here, $E(\lambda)$ is just a non-negative multiple of $D_0$, from \eqref{4.66}, and therefore, by the remark at the end of Section~\ref{transfersec2}, the function is not only non-trivial, but also non-zero for $\lambda\in\sigma_{ell}(J_T)$.

Next, consider $T\geq 2$. Letting $\lambda\rightarrow\infty, \mu\rightarrow 0$ we show that the highest-order term does not cancel. Note that from Lemma~2.1 in \cite{21} $$p_1(\lambda)\sim -\frac{a_T\lambda^{T-2}}{\prod\limits_{s=1}^{T-1}a_s}, p_2(\lambda)\sim\frac{\lambda^{T-1}}{\prod\limits_{s=1}^{T-1}a_s},$$ and using
$$\left(\begin{smallmatrix}
p_1(\lambda)&p_2(\lambda)\\
p_3(\lambda) &p_4(\lambda)\end{smallmatrix}\right)\left(\begin{smallmatrix}
\widetilde{\alpha}_1^{(j)}(\lambda)&\widetilde{\alpha}_2^{(j)}(\lambda)\\
\widetilde{\alpha}_3^{(j)}(\lambda)&\widetilde{\alpha}_4^{(j)}(\lambda)\end{smallmatrix}\right)^{-1}=\left(\begin{smallmatrix}
{\alpha}_1^{(j)}(\lambda)&{\alpha}_2^{(j)}(\lambda)\\
{\alpha}_3^{(j)}(\lambda)&{\alpha}_4^{(j)}(\lambda)\end{smallmatrix}\right)\left(\begin{smallmatrix}
0&1\\
-\frac{a_{T-j-1}}{a_{T-j}}&\frac{\lambda}{a_{T-j}}\end{smallmatrix}\right)$$
 we see from the (1,1) entry \begin{align}
\widetilde{\alpha}_3^{(j)}p_2+\widetilde{\alpha}_4^{(j)}(\mu-p_1)&=\frac{a_{T-j-1}}{a_{T-j}}\alpha_2^{(j)}\left(\widetilde{\alpha}_1^{(j)}\widetilde{\alpha}_4^{(j)}-\widetilde{\alpha}_2^{(j)}\widetilde{\alpha}_3^{(j)}\right)+\mu\widetilde{\alpha}_4^{(j)}\nonumber\\
&=\frac{a_{T-j-1}}{a_{T-j}}\alpha_2^{(j)}\deter\left(B_{T-j-1}(\lambda)\dots B_1(\lambda)\right)+\mu\widetilde{\alpha}_4^{(j)}\nonumber\\
&=\frac{a_T}{a_{T-j}}\alpha_2^{(j)}+\mu\widetilde{\alpha}_4^{(j)}, \quad j\in\{0,\dots,T-1\}\label{4.37} \end{align} where we have used $\deter B_j(\lambda)=\frac{a_{j-1}}{a_j}$.

Thus, recalling \eqref{4.66} and Lemma~\ref{4.36}, for $j\not\in\{0,T-2,T-1\}$ we have that the leading term of $D_j(\lambda)e^{i(T-j)\omega}$ as $\lambda$ tends to infinity is \begin{equation}\label{4.39}\frac{1-\mu(\lambda) p_1(\lambda)}{2(\mu(\lambda)-p_1(\lambda))p_2(\lambda)}\left(\frac{a_Ta_{T-j}\lambda^{3T-2}}{\left(\prod\limits_{s=1}^{T} a_s^{\frac{3T-2j}{T}}\right)\left(\prod\limits_{l=T-j}^{T-1} a^2_l\right)}\right).\end{equation} For $j=0$ the leading term of $D_j(\lambda)e^{i(T-j)\omega}$, as $\lambda$ tends to infinity, is \begin{equation}\label{4.40} \frac{1-\mu(\lambda) p_1(\lambda)}{2(\mu(\lambda)-p_1(\lambda))p_2(\lambda)}\left(\frac{a_T^2\lambda^{3T-2}}{\prod\limits_{s=1}^T a^3_s}\right).
\end{equation} For $j=T-2$ the leading term of $D_j(\lambda)e^{i(T-j)\omega}$, as $\lambda$ tends to infinity, is \begin{equation}\label{4.41}
\frac{1-\mu(\lambda) p_1(\lambda)}{2(\mu(\lambda)-p_1(\lambda))p_2(\lambda)}\left(\frac{a_T^3a_1^2a_2\lambda^{3T-2}}{\prod\limits_{s=1}^Ta_s^{\frac{3T+4}{T}}}\right).
\end{equation} For $j=T-1$ the leading term of $D_j(\lambda)e^{i(T-j)\omega}$, as $\lambda$ tends to infinity, is \begin{equation}\label{4.42}
\frac{1-\mu(\lambda) p_1(\lambda)}{2(\mu(\lambda)-p_1(\lambda))p_2(\lambda)}\left(\frac{a_1a_T^3\lambda^{3T-2}}{\prod\limits_{s=1}^{T} a_s^{\frac{2+3T}{T}}}\right).
\end{equation}

Since for each possible $j$ the leading term, up to an identical complex non-zero constant, is positive there is no chance of their cancellation in the sum that comprises $E(\lambda),$ and therefore the function is not identically zero.~\qedhere
\end{proof}
\begin{remark}
The function $E(\lambda)$ is algebraic, and therefore only has finitely many roots. Moreover, for $T=1$ there are no roots in $\sigma_{ell}(J_T)$. For the case $T=2$ with zero diagonal ($b_i=0$) we see by explicit calculation that \begin{multline*}E(\lambda;k_+)=\frac{e^{i\phi}(a_1+a_2\mu)}{8\lambda(a_1\mu+a_2)}\bigg[\frac{\mu\lambda^2e^{i2\omega(\lambda;k_+)}}{a_1a_2}\\+\left(\frac{\lambda^2-a_2^2}{a_1a_2}-\mu\right)\left(\frac{a_2}{a_1}+\mu\right) {e^{i\omega(\lambda;k_+)}}\bigg].\end{multline*} Then, we have that $E(\lambda; 2)=0$ if and only if $\lambda=\pm|a_1-a_2|$ which do not belong to the elliptic spectrum. Similarly, we have that $E(\lambda; 1)=0$ if and only if $\lambda=\pm(a_1+a_2)$, which again do not belong to the elliptic interval. Indeed, the points $\pm|a_1-a_2|,\pm(a_1+a_2)$ lie on the boundary.
\end{remark}

We now continue with a matrix transform. $\Gamma(\lambda):=\left(\begin{array}{cc}
0&-\frac{\overline{E(\lambda)}}{|E(\lambda)|}\\
-\frac{E(\lambda)}{|E(\lambda)|}&0\end{array}\right),$ for $\lambda\in\sigma_{ell}(J_T)$ with $E(\lambda)\neq 0$, is Hermitian, has trace zero and determinant equal to $-1$. This information dictates that $\Gamma(\lambda)$ has eigenvalues $1$ and $-1$ and is thus diagonalisable, i.e. $$\Gamma(\lambda)=W(\lambda)\left(\begin{array}{cc}
1&0\\
0&-1\end{array}\right)W^{-1}(\lambda),$$ where $W(\lambda)$ is the $2\times2$ matrix whose columns are the eigenvectors of $\Gamma$.
Consequently, we see that choosing $k_+=l_+$ \eqref{4.52} becomes
\begin{align*}
\vec{m}_{k+1}&= W(\lambda)\left[\prod\limits_{t=1}^{k} \bigg(I+\frac{{c}|E(\lambda)|}{t\sin\theta(\lambda)}\left(\begin{array}{cc}
1&0\\
0&-1\end{array}\right)\bigg)\right]W^{-1}(\lambda)\vec{m}_1~~~\\
&= W(\lambda)\left(\begin{array}{cc}
\prod\limits_{t=1}^{k}\left(1+\frac{{c}|E(\lambda)|}{t\sin\theta(\lambda)}\right)&0\\
0&\prod\limits_{t=1}^{k}\left(1-\frac{{c}|E(\lambda)|}{t\sin\theta(\lambda)}\right)\end{array}\right)W^{-1}(\lambda)\vec{m}_1\\
\sim_r \vec{m}_{k+1}&=\left(\begin{array}{cc}
\left(\widetilde{c}_1+o(1)\right) k^{\frac{{c}|E(\lambda)|}{\sin\theta(\lambda)}}&0\\
0&\left(\widetilde{c}_2+o(1)\right) k^{-\frac{{c}|E(\lambda)|}{\sin\theta(\lambda)}}\end{array}\right)\vec{m}_1,
\end{align*} for some non-zero constants $\widetilde{c}_1,\widetilde{c}_2$ depending on $\lambda\in\sigma_{ell}(J_T)$. Retracing the steps back to the original $u_n$ (as in the case of Theorem~\ref{4.44}).
This implies there exists a subordinate solution of the final system, \eqref{4.26}, asymptotically equivalent to $k^{-\left|\frac{{c}E(\lambda)}{\sin\theta(\lambda)}\right|}.$  This is in $l^2(\mathbb{N};\mathbb{C})$ if ${c}$ is large enough: $$\left|\frac{{c}E(\lambda)}{\sin\theta(\lambda)}\right|>\frac{1}{2},$$ where the value of $E(\lambda)$ is assumed to be non-zero. This completes the analysis for Case 1.

{\bf{{Case 2}}}
Here, $2\theta(\lambda)-\omega T=-2k_-\pi, \omega T\not\in\pi\mathbb{Z}$. We have from \eqref{4.67} that \begin{align*}
&\vec{h}_{k+1}=\bigg(I+\frac{{c}}{k\sin\theta(\lambda)}\sum\limits_{j=0}^{T-1}\bigg\{\left(\begin{array}{cc}
0&\overline{E_j(\lambda)}e^{i((T-j)\omega+\phi)}\\
E_j(\lambda)e^{-i((T-j)\omega+\phi)}&0\end{array}\right)\\
&-\left(\begin{array}{cc}
0&\overline{E_j(\lambda)}e^{-i(2Tk\omega+(T-j)\omega+\phi)}\\
E_j(\lambda)e^{i(2Tk\omega+(T-j)\omega+\phi)}&0\end{array}\right)\bigg\}+O\left(\frac{1}{k^2}\right)\bigg)\vec{h}_k.
\end{align*}

Then the Harris-Lutz procedure can be used again (i.e. a substitution of the from $\vec{h}_k=(I+\widetilde{H}_k)\vec{m}_k$) to get rid of the oscillating second term (as $\omega T\not\in\pi\mathbb{Z})$. Then, removing the error term using Lemma~\ref{4.34}, and for a suitable choice of $\widetilde{H}_k$ (similar to the proof of Theorem~\ref{4.44}), we have
\begin{align*}
\vec{m}_{k+1}&=\bigg(I+\frac{{c}}{k\sin\theta(\lambda)}\sum\limits_{j=0}^{T-1}\left(\begin{array}{cc}
0&\overline{E_j(\lambda)}e^{i((T-j)\omega+\phi)}\\
E_j(\lambda)e^{-i((T-j)\omega+\phi)}&0\end{array}\right)\\
&~~+O\left(\frac{1}{n^2}\right)\bigg)\vec{m}_k\\
\sim_r\vec{m}_{k+1}&=\bigg(I+\frac{{c}|\widetilde{E}(\lambda;k_-)|}{k\sin\theta(\lambda)}\left(\begin{array}{cc}
0&\frac{\overline{\widetilde{E}(\lambda;k_-)}}{|\widetilde{E}(\lambda;k_-)|}\\
\frac{\widetilde{E}(\lambda;k_-)}{|\widetilde{E}(\lambda;k_-)|}&0\end{array}\right)\bigg)\vec{m}_k,
\end{align*}
where \begin{equation}\label{4.32}\widetilde{E}(\lambda;k_-):=\sum\limits_{j=0}^{T-1}{E}_j(\lambda)e^{-i((T-j)\omega+\phi)}.\end{equation}

Observing that $\omega=\omega(\lambda;k_-)=\frac{2\theta(\lambda)+2k_-\pi}{T}$ we see $$e^{-i(T-j)\omega}=\mu^{-2}\left(\mu^{\frac{2}{T}}(\lambda)e^{\frac{i2k_-\pi}{T}}\right)^{j}.$$ As in Case 1 we can choose the branch such that $\mu(\lambda)\sim\frac{\prod\limits_{s=1}^T a_s}{\lambda^T}$ as $\lambda\rightarrow\infty$ and there exists an appropriate branch of $\left(\mu^2(\lambda)\right)^{\frac{1}{T}}$ such that $$\mu^{\frac{2}{T}}=e^{\frac{-i2\pi l_-}{T}}\left(\mu_g^2(\lambda)\right)^{\frac{1}{T}},$$ for some $l_-\in\{0,\dots, T-1\}$. Then \begin{align*}e^{i(T-j)\omega}&=\mu^{-2}(\lambda)\left(\left(\mu_g^2(\lambda)\right)^{\frac{1}{T}}(\lambda)e^{\frac{-i2 \pi l_-}{T}}e^{\frac{ik_-\pi}{T}}\right)^{j}\\
&=\mu(\lambda)^{-2}\left(\mu_g^2(\lambda)\right)^{\frac{j}{T}},
\end{align*} if $k_-$ is chosen such that $k_-=l_-$. We denote \begin{equation}\widetilde{E}(\lambda):=\widetilde{E}(\lambda;l_-).\end{equation}

\begin{lemma}\label{4.43}
The function $\widetilde{E}(\lambda)$ is algebraic and is not identically zero. Moreover we have that $$\widetilde{E}(\lambda)=e^{-i2\phi}E(\lambda).$$
\end{lemma}

\begin{proof}
We see that $$E(\lambda)=e^{i\phi}\sum\limits_{i=0}^{T-1} E_j(\lambda) e^{\frac{i2(T-j)}{T}};~~\widetilde{E}(\lambda)=e^{-i\phi}\sum\limits_{j=1}^{T-1} E_j(\lambda)e^{\frac{i2(T-j)}{T}}=e^{-i2\phi}E(\lambda).$$
That $\widetilde{E}(\lambda)$ is algebraic follows from the corresponding result for $E(\lambda)$.~\qedhere
\end{proof}

\begin{remark} For $T=1$ we have that $$E(\lambda;k_+)=\frac{e^{i\phi}}{4a_1\mu^2},~\widetilde{E}(\lambda;k_-)=\frac{e^{-i\phi}}{4a_1\mu^2}.$$
If $T>1$ then the functions $E(\lambda;k_+),\widetilde{E}(\lambda;k_-)$, besides the trivial dependence on the parameter $\phi$, depend on the frequency, $\omega$, of the perturbation, $(q_n)$, through the integer quantisation parameters, $k_+, k_-$, respectively. For $T=2$ we have that  $E(\lambda; 2)=e^{i\phi}A(\lambda), \widetilde{E}(\lambda;0)=e^{-i\phi}A(\lambda)$ where $$A(\lambda):=\frac{a_1+a_2\mu }{8a_1^2a_2\lambda(a_1\mu+a_2)}\left[\frac{1}{\mu}(\lambda^2(a_1+a_2)-a_2^3)-\mu a_1^2a_2-2a_2^2a_1+\lambda^2a_1\right],$$ and $E(\lambda;1)=e^{i\phi}B(\lambda), \widetilde{E}(\lambda;1)=e^{-i\phi}B(\lambda)$ where $$B(\lambda):=\frac{a_1+a_2\mu}{8a_1^2a_2\lambda(a_1\mu+a_2)}\left[\frac{1}{\mu}(\lambda^2(a_1-a_2)+a_2^3)+\mu a_1^2a_2+2a_2^2a_1-\lambda^2a_1\right].$$

The simple relationships between $E(\lambda;k_+)$ and $\widetilde{E}(\lambda;k_-)$ in general do not hold for $T>2$. This can be seen from the fact that \begin{equation}\label{4.57}e^{-i\left((T-j)\omega+\phi\right)}=e^{-i(T-j)\left(\frac{2\theta+2\pi k_-}{T}\right)-i\phi}=\mu^{\frac{2(T-j)}{T}}e^{\frac{2ij\pi k_-}{T}}e^{-i\phi},\end{equation} for $k_-\in\{0,\dots,T-1\}$ and \begin{equation}\label{4.58}e^{i\left((T-j)\omega+\phi\right)}=e^{i(T-j)\left(\frac{2\pi k_+-2\theta}{T}\right)+i\phi}={\mu^{\frac{2(T-j)}{T}}}e^{-\frac{2ij\pi k_+}{T}}e^{i\phi},\end{equation} for $k_+\in\{1,\dots,T\}$.
\end{remark}

By using the same diagonalisation argument as described in Case 1 we see that \eqref{4.52} becomes
\begin{align*}
\vec{m}_{k+1}&= W(\lambda)\left[\prod\limits_{t=1}^{k} \bigg(I+\frac{{c}|\widetilde{E}(\lambda)|}{t\sin\theta(\lambda)}\left(\begin{array}{cc}
1&0\\
0&-1\end{array}\right)\bigg)\right]W^{-1}(\lambda)\vec{m}_1\\
&=W(\lambda)\left(\begin{array}{cc}
\prod\limits_{t=1}^{k}\left(1+\frac{{c}|\widetilde{E}(\lambda)|}{t\sin\theta(\lambda)}\right)&0\\
0&\prod\limits_{t=1}^{k}\left(1-\frac{{c}|\widetilde{E}(\lambda)|}{t\sin\theta(\lambda)}\right)\end{array}\right)W^{-1}(\lambda)\vec{m}_1\\
\sim_r \vec{m}_{k+1}&=\left(\begin{array}{cc}
\left(\widetilde{c}_3+o(1)\right) k^{\frac{{c}|\widetilde{E}(\lambda)|}{\sin\theta(\lambda)}}&0\\
0&\left(\widetilde{c}_4+o(1)\right) k^{-\frac{{c}|\widetilde{E}(\lambda)|}{\sin\theta(\lambda)}}\end{array}\right)\vec{m}_1,
\end{align*} for some non-zero constants $\widetilde{c}_3,\widetilde{c}_4$ depending on $\lambda\in\sigma_{ell}(J_T)$.
As in Case 1, this implies there exists a subordinate solution of the final system, \eqref{4.26}, asymptotically equivalent to $k^{-\big|\frac{{c}\widetilde{E}(\lambda)}{\sin\theta(\lambda)}\big|}.$ This is in $l^2(\mathbb{N};\mathbb{C})$ if ${c}$ is large enough: $$\bigg|\frac{{c}\widetilde{E}(\lambda)}{\sin\theta(\lambda)}\bigg|>\frac{1}{2},$$ where the value of $\widetilde{E}(\lambda)$ is assumed to be non-zero. This completes the analysis for Case 2.

{\bf{{Case 3}}} Here $2\theta(\lambda)+\omega T=2k_+\pi,2\theta(\lambda)-\omega T=-2k_-\pi$ which implies $\theta(\lambda)=\frac{\pi}{2}, k_-=k_+-1$ and $\omega T=(2k_+-1)\pi$, $k_+\in\{1,\dots,T\}$. Thus, there are no oscillating terms and another application of the Harris-Lutz procedure is not needed. Then, from \eqref{4.67} and removing the error term using Lemma~\ref{4.34}, we have
\small\begin{align*}
\vec{h}_{k+1}&=\bigg(I-\frac{{c}}{k}\sum\limits_{j=0}^{T-1}\left(\begin{smallmatrix}
0&\overline{E_j(\lambda)}\left(e^{-i((T-j)\omega+\phi)}-e^{i((T-j)\omega+\phi)}\right)\\
E_j(\lambda)\left(e^{i((T-j)\omega+\phi)}-e^{-i((T-j)\omega+\phi)}\right)&0\end{smallmatrix}\right)\\
&~~+O\left(\frac{1}{k^2}\right)\bigg)\vec{h}_k\\
\sim_r\vec{h}_{k+1}&=\bigg(I-\frac{{c}}{k}\sum\limits_{j=0}^{T-1}\left(\begin{smallmatrix}
0&\overline{E_j(\lambda)}\left(e^{-i((T-j)\omega+\phi)}-e^{i((T-j)\omega+\phi)}\right)\\
E_j(\lambda)\left(e^{i((T-j)\omega+\phi)}-e^{-i((T-j)\omega+\phi)}\right)&0\end{smallmatrix}\right)\bigg)\vec{h}_k\\
&=\bigg(I-\frac{{c}|\widetilde{\widetilde{E}}(k_+)|}{k}\sum\limits_{j=0}^{T-1}\left(\begin{array}{cc}
0&\frac{\overline{\widetilde{\widetilde{{E}}}(k_+)}}{|\widetilde{\widetilde{E}}(k_+)|}\\
\frac{\widetilde{\widetilde{E}}(k_+)}{|\widetilde{\widetilde{E}}(k_+)|}&0\end{array}\right)\bigg)\vec{h}_k,
\end{align*}\normalsize
where \begin{equation}\label{4.33}\widetilde{\widetilde{E}}(k_+):=\sum\limits_{j=0}^{T-1} E_j\left(\theta^{-1}\left(\frac{\pi}{2}\right)\right)\left(e^{i(j\omega+\phi)}-e^{i(-j\omega+\phi)}\right)\end{equation} with the simplification $e^{i\omega T}=-1$ and the error term was removed using Lemma~\ref{4.34}.

\begin{remark}
Due to the very specific relations between $k_+$ and $k_-$ for this case, it is not possible to imitate the technique employed in $E(\lambda),\widetilde{E}(\lambda)$, to eliminate the roots of unity that arise from the expression $e^{\pm i(T-j)\omega}$, and show that the expression, $\widetilde{\widetilde{E}}(\lambda)$, is not identically zero simply by looking at the signs of the leading terms. However, since this case only arises when $\theta(\lambda)=\frac{\pi}{2}$, then by the strict monotonicity of $\theta(\lambda)$ on the elliptic interval (see, for example, Equation 3.10 in \cite{7.40}) there is only one $\lambda$ in each band of essential spectrum that is possibly excluded. Indeed, for the case $T=1$ we see that $\omega=\pi$ and $$\widetilde{\widetilde{E}}(k_+)=\frac{-1}{4a_1e^{i\phi}}\left(e^{2i\phi}-1\right),$$ which implies that providing $\phi\not\equiv 0\mod \pi$ the function $\widetilde{\widetilde{E}}(\lambda)$ is non-zero and the technique is therefore applicable. Note that in the current case if $\phi\equiv 0\mod \pi$, then $q_n\equiv 0$, so there are no subordinate solutions.
\end{remark}

Now, using the same diagonalisation argument described in Case 1, and defining $W:=W(\theta^{-1}(\frac{\pi}{2}))$ we obtain

\begin{align*}
\vec{h}_{k+1}&= W\left[\prod\limits_{t=1}^{k} \bigg(I+\frac{{c}|\widetilde{\widetilde{E}}(k_+)|}{t}\left(\begin{array}{cc}
1&0\\
0&-1\end{array}\right)\bigg)\right]W^{-1}\vec{h}_{1}\\
&= W\left(\begin{array}{cc}
\prod\limits_{t=1}^{k}\left(1+\frac{{c}|\widetilde{\widetilde{E}}(k_+)|}{t}\right)&0\\
0&\prod\limits_{t=1}^{k+1}\left(1-\frac{{c}|\widetilde{\widetilde{E}}(k_+)|}{t}\right)\end{array}\right)W^{-1}\vec{h}_1\\
\sim_r \vec{h}_{k+1}&= \left(\begin{array}{cc}
\left(\widetilde{c}_5+o(1)\right) k^{{c}|\widetilde{\widetilde{E}}(k_+)|}&0\\
0&\left(\widetilde{c}_6+o(1)\right) k^{-{c}|\widetilde{\widetilde{E}}(k_+)|}\end{array}\right)\vec{h}_1,
\end{align*} for non-zero constants $\widetilde{c}_5,\widetilde{c}_6$ depending on $\lambda\in\sigma_{ell}(J_T)$. As in Cases 1 and 2, this implies there exists a subordinate solution of the final system asymptotically equivalent to $k^{-{c}|\widetilde{\widetilde{E}}(k_+)|}.$ This is in $l^2(\mathbb{N};\mathbb{C})$ if ${c}$ is large enough: $$\big|{c}\widetilde{\widetilde{E}}(k_+)\big|>\frac{1}{2},$$ where $\widetilde{\widetilde{E}}(k_+)$ is assumed to be non-zero.

 Thus regardless of the case, there always exists a subordinate solution, providing a suitable Wigner-von Neumann potential is chosen and the corresponding value of $E(\lambda;k_+),\widetilde{E}(\lambda;k_-),\widetilde{\widetilde{E}}(k_+)$ is non-zero. The subordinate solution is in $l^2$ if we choose the constant $c$ large enough. This proves Theorem~\ref{transfer2}.

 \begin{remark}
 We expect that any $\lambda\in\sigma_{ell}(J_T)$ is not simultaneously a root of $E(\lambda;k_+),\widetilde{E}(\lambda;k_-)$, for all $k_+\in\{1,\dots,T\}, k_-\in\{0,\dots,T-1\}.$ Therefore, whenever the quantisation condition is satisfied a subordinate solution should exist.
 \end{remark}

The above result, however, gives only the chance to prove that a potential of the form, $q_n=\frac{c\sin(n\omega+\phi)}{n}$ embeds an eigenvalue, rather than produce only a subordinate solution. In order for it to be a true eigenvector, the initial conditions encoded in the periodic Jacobi operator must also be satisfied. This leads to:

 \begin{thm}\label{transfer5}
 Let $\lambda\in\sigma_{ell}(J_T).$ If $\theta(\lambda)\neq\frac{\pi}{2}$ choose $\omega$ s.t. $\omega T+2\theta(\lambda)= 2 \pi k_+$ and assume $E(\lambda;k_+)\neq 0$. Then construct a potential $(q_n)$ with $$q_n=\frac{{c}\sin(n\omega+\phi)}{n},~~n\geq 3$$ for arbitrary $\phi\in\mathbb{R}$. One can choose a suitable ${c}$ and real values $q_1, q_2$ such that $$\lambda\in\sigma_p(J_T+Q)$$ where $\sigma_p$ is the point spectrum and $Q$ is a diagonal matrix with entries $(q_n)$.
 \end{thm}

\begin{remark}
Similar results can proved if one of the other quantisation conditions,~\eqref{4.30}, is satisfied.
\end{remark}
\begin{proof} Theorem~\ref{transfer2} gives $(u_n),(q_n)$ such that $$a_{n-1}u_{n-1}+a_nu_{n+1}+(q_n+b_n-\lambda)u_n=0$$ for $n\geq 3$. There are two cases: \begin{enumerate}
\item
If $u_2\neq 0$ then defining $q_2:=\frac{-\lambda u_2-a_2u_3-a_1u_1-b_2u_2}{u_2}$ with $u_1:=-\frac{a_1u_2}{q_1+b_1-\lambda}$, with $q_1$ as a free
parameter and not equal to $\lambda-b_1$, ensures all conditions are satisfied.
\item
If $u_2=0$ then defining $u_1:=-\frac{a_2u_3}{a_1}$ and $q_1:=\lambda-b_1$, with $q_2$ as a free parameter, ensures all conditions are
satisfied.
\end{enumerate}
See the proof of Theorem~7.1 in \cite{21} for more details.~\qedhere
\end{proof}

\section{Multiple subordinate solutions}\label{transfersec6}
Here we extend Theorem~\ref{transfer2} to construct subordinate solutions for a (possibly infinite) set of spectral parameters belonging to the generalised interior of the essential spectrum.
\begin{thm}\label{transfer3}
 Let $S\subseteq \mathbb{N}$ and $(\lambda_i)_{i\in S}$ be a sequence of numbers belonging to $\sigma_{ell}(J_T)$. Assume $\theta(\lambda_i)\neq \frac{\pi}{2}$ and $E(\lambda_i)\neq 0$ for all $i\in S$, with \begin{equation}\label{4.92}q_n^{(i)}:=\frac{\sin(n\omega_i+\phi_i)}{n},\end{equation} where $\omega_i$ is such that $T\omega_i+2\theta(\lambda_i)=2k^{(i)}_+\pi,$ for a suitably chosen integer $k_+^{(i)}$. Then, there exists a real strictly positive sequence $(c_i)_{i\in S}$ belonging to $l^1(\mathbb{N})$ such that for the potential, $(q_n)$,  \begin{equation}\label{mult1}q_n:=\sum\limits_{i\in S} c_iq_n^{(i)}=\sum\limits_{i\in S}\frac{c_i\sin(n\omega_i+\phi_i)}{n},\end{equation} for arbitrary $\phi_i\in\mathbb{R}$, there are subordinate solutions, $\underline{u}^{(i)}:=\left(u_n^{(i)}\right)_{n\geq 1}$, to the recurrence equations $$a_{n-1}u_{n-1}^{(i)}+(b_n+q_n)u_n^{(i)}+a_{n}u_{n+1}^{(i)}=\lambda_iu_n^{(i)}, n\geq 2, i\in S.$$
 \end{thm}

\begin{remark}
The reader should observe that there is no rational dependence condition between the $\theta(\lambda),\lambda\in(\lambda_i)_{i\in S}$, like in some results (see, for example, Theorem~1 in \cite{9g}). Indeed, our only constraint is that $E(\lambda_i)\neq 0$ and since the function is algebraic there are only finitely many roots and therefore finitely many points in the elliptic spectrum where the technique fails. (For the periods $T=1$ and $T=2$ we have seen that the function $E(\lambda)$ has no roots in $\sigma_{ell}(J_T)$ and therefore there are no restrictions for these two cases.) Moreover, the frequency, $\omega_i$ used to define the potential has no dependency on any other $\lambda_i$ than that with which it satisfies the resonance conditions.
\end{remark}

\begin{remark}
To simplify notation, unless explicitly mentioned we will assume $S=\mathbb{N}$ as this is the most general and interesting case. All other cases can be proven in the same way. Later, in Theorem~\ref{4.27}, for the case $S=\{1,2\}$, we deal with the initial conditions and establish explicit $\underline{u}^{(i)}\in l^2,(q_n)$ such that $$(J_T+Q)\underline{u}^{(i)}=\lambda_i\underline{u}^{(i)}$$ for each $i\in\{1,2\}$ and where $Q$ is a diagonal matrix with entries $q_n$.
\end{remark}

  The aim is to consider an arbitrary $\lambda_t\in(\lambda_i)_{i\in S}$ and show that the new perturbation, $(q_n)$, still produces a subordinate solution for $\lambda_t$. Note that each $\lambda_i$ will now be associated to an eigenvalue, $\mu(\lambda_i)$, of the monodromy matrix, where $\mu(\lambda_{i})=e^{i\theta(\lambda_i)}$. Moreover, since the explicit nature of $(q_n)$ in the single eigenvalue case is not discussed until the section dealing with the Harris-Lutz procedure in the proof of Theorem~\ref{transfer2}, this means that the results of Sections~\ref{transfersec1} and \ref{transfersec2} are still applicable here. Moreover, we see that by choosing $(c_l)_{l\in S}$ such that $\sum\limits_{l=1}^\infty c_l<\infty$ then \begin{align*} |q_n|=\left|\sum\limits_{l=1}^\infty c_l \frac{\sin(n\omega_l+\phi_l)}{n}\right|\leq \sum\limits_{l=1}^\infty \frac{c_l}{n}=O\left(\frac{1}{n}\right). \end{align*} The details at the end of Section~\ref{transfersec3} follow similarly to before except now we must use a more detailed version of Proposition~\ref{zyg}:

  \begin{lemma}\label{4.73}
Let $\alpha\in\mathbb{R}, \alpha\not\in 2\pi\mathbb{Z}$. Then $$\left|\sum\limits_{k=n}^\infty \frac{e^{ik\alpha}}{k}\right|\leq \frac{2}{n\left|e^{i\alpha}-1\right|}= \frac{1}{n\left|\sin(\frac{\alpha}{2})\right|}.$$
\end{lemma}

\begin{proof}
Define $\widetilde{g}_n:=\sum\limits_{k=n}^\infty \frac{e^{ik\alpha}}{k}$, which exists by Proposition~\ref{zyg}. Consequently,
\begin{equation}\label{4.72}
\widetilde{g}_ne^{i\alpha}=\sum\limits_{k=n}^\infty \frac{e^{i(k+1)\alpha}}{k}=\sum\limits_{k=n}^\infty \frac{e^{i(k+1)\alpha}}{k+1}+\sum\limits_{k=n}^\infty \frac{e^{i(k+1)\alpha}}{k(k+1)}=\widetilde{g}_{n+1}+\sigma_n,
\end{equation} where $\sigma_n:=\sum\limits_{k=n}^\infty \frac{e^{i(k+1)\alpha}}{k(k+1)}$ and $|\sigma_n|\leq \frac{1}{n}.$ Then, by~\eqref{4.72} \begin{equation*}(e^{i\alpha}-1)\widetilde{g}_n=\widetilde{g}_{n+1}-\widetilde{g}_n+\sigma_n=-\frac{e^{in\alpha}}{n}+\sigma_n,
\end{equation*} which implies $|e^{i\alpha}-1| |\widetilde{g}_n|\leq \frac{2}{n}.$ Thus $$|\widetilde{g}_n|=\left|\sum\limits_{k=n}^\infty \frac{e^{ik\alpha}}{k}\right|\leq \frac{2}{n\left|e^{i\alpha}-1\right|}.~\qedhere$$
\end{proof}
This gives the following corollary which will be used repeatedly throughout this section of the paper.

\begin{cor}\label{4.88}
Let $\alpha\in\mathbb{R}, \alpha\not\in 2\pi\mathbb{Z}$. Then for $n_1,n_2\in\mathbb{N}, n_2>n_1$
$$\left|\sum\limits_{k=n_1}^{n_2} \frac{e^{ik\alpha}}{k}\right|\leq\left(\frac{1}{n_1}+\frac{1}{n_2}\right)\frac{1}{|\sin\left(\frac{\alpha}{2}\right)|}.$$

\end{cor}

\begin{proof}
Observe that \begin{align*}
\left|\sum\limits_{k=n_1}^{n_2}\frac{e^{ik\alpha}}{k}\right|&=\left|\sum\limits_{k=n_1}^\infty \frac{e^{ik\alpha}}{k}-\sum\limits_{k'=n_2+1}^\infty \frac{e^{ik'\alpha}}{k'}\right|\leq \left|\sum\limits_{k=n_1}^\infty \frac{e^{ik\alpha}}{k}\right|+\left|\sum\limits_{k'=n_2+1}^\infty \frac{e^{ik'\alpha}}{k'}\right|\\
&\leq \frac{1}{n_1\left|\sin\left(\frac{\alpha}{2}\right)\right|}+\frac{1}{n_2\left|\sin\left(\frac{\alpha}{2}\right)\right|}=\left(\frac{1}{n_1}+\frac{1}{n_2}\right)\frac{1}{|\sin\left(\frac{\alpha}{2}\right)|},
\end{align*} where the final inequality is a consequence of Lemma~\ref{4.73}.~\qedhere
\end{proof}

To apply the Harris-Lutz transformation in this case we define $\vec{h}_k$ such that $\vec{g}_k=(I+G_k)\vec{h}_k$ with $\|G_k\|=O\left(\frac{1}{k}\right)$. From the analogue of \eqref{4.24} this gives \begin{equation}\vec{h}_{k+1}=\left(I-G_{k+1}+G_k+T_1(k)+T_2(k)+O\left(\frac{1}{k^2}\right)\right)\vec{h}_k,\end{equation} where  \begin{equation}T_1(k):=\frac{1}{i\sin\theta(\lambda_t)}\sum\limits_{l=1}^\infty\sum\limits_{j=0}^{T-1} \frac{c_lq^{(l)}_{(k+1)T-j}}{a_{T-j}}\left(\begin{array}{cc}
-\overline{C_j(\lambda_t)}&0\\
0&C_j(\lambda_t)\end{array}\right),\end{equation}\begin{equation} T_2(k):=\frac{1}{i\sin\theta(\lambda_t)}\sum\limits_{l=1}^\infty \sum\limits_{j=0}^{T-1} \frac{c_lq^{(l)}_{(k+1)T-j}}{a_{T-j}}\left(\begin{array}{cc}
0&-\overline{D_j(\lambda_t)}~\overline{\mu(\lambda_t)}^{2k}\\
D_j(\lambda_t)\mu(\lambda_t)^{2k}&0\end{array}\right)\end{equation} with $C_j(\lambda_t), D_j(\lambda_t)$ as defined in \eqref{4.74} and \eqref{4.66}. Then $T_1(k),T_2(k)=O\left(\frac{1}{k}\right)$ due to the condition that $(c_l)_{l\in S}\in l^1$. In addition, define $G^N_k:=-\sum\limits_{m=k}^N T_1(m)$. Consequently, for $N_1,N_2$ large enough with $N_2>N_1>k$ we have \begin{align}
&\left\|G_k^{N_2}-G_k^{N_1}\right\|=\left\|\sum\limits_{m=N_1}^{N_2}T_1(m)\right\|=\left|\sum\limits_{m=N_1}^{N_2}\frac{1}{\sin\theta(\lambda_t)}\sum\limits_{l=1}^\infty \sum\limits_{j=0}^{T-1} c_l \frac{q_{(m+1)T-j}^{(l)}}{a_{T-j}} C_j(\lambda_t)\right|\nonumber\\
&=\frac{1}{|\sin\theta(\lambda_t)|}\left|\sum\limits_{l=1}^\infty c_l\sum\limits_{j=0}^{T-1} \frac{C_j(\lambda_t)}{a_{T-j}}\sum\limits_{m=N_1}^{N_2}\frac{\sin(((m+1)T-j)\omega_l+\phi_l)}{(m+1)T-j}\right|.\nonumber
\end{align} Then, using that $$\frac{1}{(m+1)T-j}\leq\frac{1}{(m+1)T}+\frac{j}{m^2T^2}$$ and $$\sum\limits_{m=N_1}^{N_2} \frac{1}{m^2}\leq \sum\limits_{m=N_1}^\infty \frac{1}{m^2}\leq \frac{1}{N_1},$$ we obtain
\begin{align}
&\left\|G_k^{N_2}-G_k^{N_1}\right\|\leq \frac{K_t}{N_1}+\frac{1}{|\sin\theta(\lambda_t)|}\sum\limits_{l=1}^\infty c_l\left|\sum\limits_{j=0}^{T-1} \frac{C_j(\lambda_t)}{a_{T-j}}\im\sum\limits_{m=N_1}^{N_2}\frac{e^{i(((m+1)T-j)\omega_l+\phi_l)}}{(m+1)T}\right|\nonumber\\
&\leq \frac{K_t}{N_1}+\widetilde{K}_t\sum\limits_{l=1}^\infty c_l\left|\sum\limits_{m'=N_1+1}^{N_2+1} \frac{e^{im'T\omega_l}}{m'}\right|\leq \frac{K_t}{N_1}+\frac{2\widetilde{K}_t}{TN_1}\sum\limits_{l=1}^\infty \frac{c_l}{\left|\sin\left(\frac{T\omega_l}{2}\right)\right|},\label{4.80}
\end{align} where $K_t:=T\widetilde{K}_t\sum\limits_{l=1}^\infty c_l,~\widetilde{K}_t:=\frac{\max\limits_{j\in\{0,\dots,T-1\}}\left| C_j(\lambda_t)\right|}{|\sin\theta(\lambda_t)|\min\limits_{j\in\{0,\dots,T-1\}} a_j}$, and the final inequality follows from Corollary~\ref{4.88} and the strictly positive sequence $(c_l)_{l\in S}$ being chosen such that \begin{equation}\label{4.89}\sum\limits_{l=1}^\infty \frac{c_l}{|\sin\left(\frac{T\omega_l}{2}\right)|}<\infty.\end{equation} Clearly, $\left(G_k^N\right)_N$ is a Cauchy-sequence  and therefore we have that the limit $G_k:=\lim_{N\rightarrow\infty} G^N_k$ exists.

For $N$ sufficiently large we can employ similar techniques to establish that $$\left\|G_k^N\right\|\leq \frac{K_t}{k}+\frac{2\widetilde{K}_t}{Tk}\sum\limits_{l=1}^\infty c_l \frac{1}{\left|\sin\left(\frac{T\omega_l}{2}\right)\right|}=O\left(\frac{1}{k}\right)$$ is uniformly bounded in $N$ for any value of the parameter $t$. Consequently, the limit $G_k$ is also bounded for any value of the parameter $t$, providing the strictly positive sequence $(c_l)_{l\in S}$ is chosen to decay fast enough.

\begin{remark}
It should be stressed that any real strictly positive sequence which satisfies \eqref{4.89} will suffice to make the above Harris-Lutz procedure valid.
\end{remark}

The Harris-Lutz procedure is well-defined and we remove the $T_1(k)$ term like in the single eigenvalue case (i.e. $G_{k+1}-G_k=T_1(k)$). The analogue of Section~\ref{transfersec4} becomes

 \begin{align}
&\vec{h}_{k+1}=\left(I+T_2(k)+O\left(\frac{1}{k^2}\right)\right)\vec{h}_k\nonumber\\
&=\bigg(I+\frac{1}{i\sin\theta(\lambda_t)}\sum\limits_{j=0}^{T-1}\sum\limits_{l=1}^\infty\frac{{c_l}\sin({((k+1)T-j)\omega_l+\phi_l})}{a_{T-j}((k+1)T-j)}\nonumber\\
&\times\left(\begin{smallmatrix}
0&-\overline{D_j(\lambda_t)}~\overline{\mu(\lambda_t)}^{2k}\\
D_j(\lambda_t)\mu^{2k}(\lambda_t)&0\end{smallmatrix}\right)+O\left(\frac{1}{k^2}\right)\bigg)\vec{h}_k\nonumber\\
&=\bigg(I+\frac{1}{i\sin\theta(\lambda_t)}\sum\limits_{j=0}^{T-1}\sum\limits_{l=1}^\infty\frac{{c_l}e^{i(((k+1)T-j)\omega_l+\phi_l)}-{c_l}e^{-i(((k+1)T+j)\omega_l+\phi_l)}}{2ia_{T-j}((k+1)T+j)}\nonumber\\
&~~~~~~~~~~\times\left(\begin{smallmatrix}
0&-\overline{D_j(\lambda_t)}\overline{\mu}^{2k}(\lambda_t)\\
D_j(\lambda_t)\mu^{2k}(\lambda_t)&0\end{smallmatrix}\right)+O\left(\frac{1}{k^2}\right)\bigg)\vec{h}_k\nonumber\\
&=\bigg(I+\frac{1}{(k+1)\sin\theta(\lambda_t)}\sum\limits_{l=1}^{\infty}c_l\bigg\{\left(\begin{smallmatrix}
0&-\overline{d_l(\lambda_t)}+\overline{f_l(\lambda_t)}\\
-d_l(\lambda_t)+f_l(\lambda_t)&0\end{smallmatrix}\right)+O\left(\frac{1}{k^2}\right)\bigg)\vec{h}_k\label{4.53}
\end{align} where $$d_l(\lambda_t):= e^{ik[2\theta(\lambda_t)+T\omega_l]}\sum\limits_{j=0}^{T-1}E_j(\lambda_t)e^{i((T-j)\omega_l+\phi_l)}$$ and $$f_l(\lambda_t):=e^{-ik[T\omega_l-2\theta(\lambda_t)]}\sum\limits_{j=0}^{T-1}{E}_j(\lambda_t)e^{-i((T-j)\omega_l+\phi_l)}$$ with $E_j(\lambda_t)$ as defined previously.

 To establish that the perturbation affects the asymptotics, we observe that for $l=t$ there is resonance between the frequency of the oscillation and the quasi-momentum. In particular, for a suitable choice of $k_+^{(i)}$ (see Section~\ref{transfersec5}) we have $$c_t\sum\limits_{j=0}^{T-1} E_j(\lambda_t)e^{i((T-j)\omega_l+\phi_l)}=c_tE(\lambda_t),$$ which is non-zero by the conditions assumed in the statement of the theorem. Furthermore, it is possible for other resonance to occur when $l\neq t$ and this is discussed in the following lemma.

 \begin{lemma}\label{4.85}
Let $\lambda_\alpha\sim\lambda_\beta$ denote when $2\theta(\lambda_\alpha)+T\omega_{\beta}\in2\pi\mathbb{Z}$ or $2\theta(\lambda_\alpha)-T\omega_{\beta}\in2\pi\mathbb{Z}$ for $\lambda_\alpha,\lambda_\beta$ from the set $(\lambda_i)_{i\in S}$ given in Theorem~\ref{transfer3} and where $\omega_i$ is such that $2\theta(\lambda_i)+T\omega_i\in 2\pi\mathbb{Z}$. Then $\sim$ is an equivalence relation; in particular, the set $\left\{\lambda_i~|~i\in S\right\}$ can be partitioned into equivalence classes. Moreover, each class has at most $2T$ elements.
\end{lemma}

\begin{proof}
First observe that \begin{align*}\lambda_\alpha\sim\lambda_\beta&\iff 2\theta(\lambda_\alpha)+ T\omega_\beta\in 2\pi\mathbb{Z}~{\rm or}~2\theta(\lambda_\alpha)-T\omega_\beta\in 2\pi\mathbb{Z}\\
&\iff \theta(\lambda_\alpha)-\theta(\lambda_\beta)\in\pi \mathbb{Z}~{\rm or}~\theta(\lambda_\alpha)+\theta(\lambda_\beta)\in\pi \mathbb{Z}, \end{align*} using that $T\omega_i=2z_1\pi-2\theta(\lambda_i)$ for some $z_1\in\mathbb{Z}$ by the conditions assumed in the theorem. As $\theta(\lambda)\in(0,\pi)$, this implies $$\lambda_\alpha\sim\lambda_\beta \iff \theta(\lambda_\alpha)-\theta(\lambda_\beta)=0~{\rm or}~\theta(\lambda_\alpha)+\theta(\lambda_\beta)=\pi,$$ which is clearly an equivalence relation. Moreover, as $\theta(\lambda)$ is strictly monotonic on each band of $\sigma_{ell}(J_T)$ (see, for example, 3.10 in \cite{7.40}) each band of essential spectrum contributes at most 2 elements to the equivalence class. Given that there are at most $T$ bands of essential spectrum for a period-$T$ Jacobi operator, then there at most $2T$ elements for the equivalence class in total. See Figure~\ref{4.87} for an illustration.\qedhere
\end{proof}

\begin{figure}[h]
\centering
\begin{tikzpicture}
        \draw [ultra thick, black, -] (-5,0) -- (-2,0)
         node [thick, right, red] at (-3,0) {\large$X$}
         node [thick, left, blue] at (-4,0) {\large$\widetilde{X}$};
         \draw[ultra thick, black, -] (-1,0) -- (2,0)
         node [thick, left, red] at (0,0) {\large$X$}
         node [thick, right, blue] at (1,0) {\large$\widetilde{X}$};
            \draw[ultra thick, black, -] (3,0) -- (6,0)
            node [thick, right,red ] at (5,0) {\large$X$}
            node [thick, left, blue] at (4,0) {\large$\widetilde{X}$};
\end{tikzpicture}
\caption{In this example, for a period-3 Jacobi operator, the three thick horizontal lines denote the three bands of essential spectrum. The possible resonant cases, $\lambda_i$, for a particular point $\lambda$ in the generalised interior are represented by $X$ and $\widetilde{X}$, for those $\lambda_i$ such that $\theta(\lambda_i)=\theta(\lambda)$, and $\theta(\lambda_i)=\pi-\theta(\lambda)$, respectively. Note that the six points are only candidate elements of the equivalence class, since it still remains to check for each whether it is also an element of the sequence $(\lambda_i)_{i\in S}$. }
\label{4.87}
\end{figure}
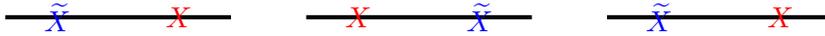

From Lemma~\ref{4.85} we see there are only finitely many resonating terms for each fixed $t$. Later, it will be shown that for an appropriate choice of $(c_j)_{j\in S}$ these finitely many resonance terms cannot cancel, but for now we focus on removing the infinitely many non-resonant terms from the consideration of the asymptotics for \eqref{4.53} using the Harris-Lutz technique.

Define, for each $t\in\mathbb{N}$, $$I_t^+:=\{n\in\mathbb{N} |2\theta(\lambda_t)+T\omega_n \in 2\pi\mathbb{Z}\},~~I_t^-:=\{{n}\in\mathbb{N} | 2\theta(\lambda_t)-T\omega_{{n}}\in 2\pi\mathbb{Z}\}.$$ For the case that $2\theta(\lambda_t)+T\omega_t\in2\pi\mathbb{Z}$, by Corollary~\ref{4.88} we have for any $M$ $$ \left| \sum\limits_{k=1}^M \sum\limits_{\substack{l=1,\\ l\not\in I_t^+}}^\infty c_l \frac{e^{-ik[2\theta(\lambda_t)+T\omega_l]}}{k}\right|\leq\sum\limits_{\substack{l=1,\\l\not\in I_t^+}}^\infty c_l \left|\sum\limits_{k=1}^M \frac{e^{ikT[\omega_t-\omega_l]}}{k}\right|\leq 2 \sum\limits_{\substack{l=1,\\l\not\in I_t^+}}^\infty \frac{c_l}{\left|\sin\left(\frac{T\omega_t-T\omega_l}{2}\right)\right|}.$$  Thus \begin{equation} \label{4.75} \left| \sum\limits_{k=1}^\infty \sum\limits_{\substack{l=1,\\ l\not\in I_t^+}}^\infty c_l \frac{e^{-ik[2\theta(\lambda_t)+T\omega_l]}}{k}\right|\leq 2\sum\limits_{\substack{l=1,\\l\not\in I_t^+}}^\infty \frac{c_l}{\left|\sin\left(\frac{T\omega_t-T\omega_l}{2}\right)\right|}.\end{equation} Similarly
 \begin{equation}\label{4.76}\left| \sum\limits_{k=1}^\infty \sum\limits_{\substack{l=1,\\l\not\in I_t^-}}^\infty c_l  \frac{e^{ik[T\omega_l-2\theta(\lambda_t)]}}{k}\right|\leq 2\sum\limits_{\substack{l=1,\\l\not\in I_t^-}}^\infty \frac{c_l}{\left|\sin\left(\frac{T\omega_t+T\omega_l}{2}\right)\right|}.\end{equation} Now an upper-bound needs to be established for \eqref{4.75} and \eqref{4.76} that is uniform in the parameter $t$. This is achieved by defining a new positive sequence \begin{equation}\label{4.90} b_k:=\min\left\{\min\limits_{j\in\{1,\dots,k-1\}\setminus I_k^+}\left|\sin\left(\frac{T\omega_k-T\omega_j}{2}\right)\right|,\min\limits_{j\in\{1,\dots,k-1\}\setminus I_k^-}\left|\sin\left(\frac{T\omega_k+T\omega_j}{2}\right)\right|\right\},\end{equation} $k>2T$, and observing that for all $t\in\mathbb{N}$ we have
$\left|\sin\left(\frac{T\omega_t-T\omega_l}{2}\right)\right|\geq b_l,$ for $l>\max\{t,2T\}, l\not \in I_t^+$, and $\left|\sin\left(\frac{T\omega_t+T\omega_l}{2}\right)\right|\geq b_l$ for $l>\max\{t,2T\}, l\not \in I_t^-$.
  \begin{remark}
  The motivation behind letting $k>2T$ follows from the fact that the formal definition of $b_k$ requires taking the minimum over a set that is non-empty. Since there are at most $2T$ possible $j$ where resonance occurs, and these instances are excluded from our consideration, we must ensure that we are taking the minimum over a set that has more than $2T$ entries to guarantee at least one entry in the set.
  \end{remark}

  The initial terms of the series being finite, we focus on the tail and see that \begin{align}\sum\limits_{\substack{l=2T+1,\\l\not\in I_t^+}}^\infty \frac{c_l}{\left|\sin\left(\frac{T\omega_t-T\omega_l}{2}\right)\right|}\leq \sum\limits_{\substack{l=2T+1\\l\not\in I_t^+}}^\infty \frac{c_l}{b_l}, \label{4.79}\end{align}
which is convergent providing the sequence $(c_j)_{j\in S}$ is chosen such that \begin{equation}\label{4.91} \sum\limits_{l} \frac{c_l}{b_l}< \infty.\end{equation} Similarly, we have
\begin{align}\sum\limits_{\substack{l=2T+1,\\l\not\in I_t^-}}^\infty \frac{c_l}{\left|\sin\left(\frac{T\omega_t+T\omega_l}{2}\right)\right|}&\leq \sum\limits_{\substack{l=2T+1,\\l\not\in I_t^-}}^\infty \frac{c_l}{b_l},\label{4.81}\end{align} which is also finite providing $(c_j)_{j\in S}$ satisfies the same conditions as in \eqref{4.91}. If instead $2\theta(\lambda_t)-\omega_t T\in2\pi\mathbb{Z}$ we obtain similar estimates.

 It is now possible to remove the non-oscillating terms using a well-defined Harris-Lutz transformation, i.e. a substitution of the form $\vec{h}_{k+1}=(I+H_k)\vec{m}_k$ with $\left\|H_k\right\|=O\left(\frac{1}{k}\right)$, which gives $$\vec{m}_{k+1}=\left(I-H_{k+1}+H_k+T_3(k)+T_4(k)+O\left(\frac{1}{k^2}\right)\right)\vec{m}_{k+1},$$ where \begin{align*}T_3(k)&=\frac{-1}{(k+1)\sin\theta(\lambda_t)}\bigg\{\left(\begin{smallmatrix}
0&\sum\limits_{l=1, l\not\in{I_t^+}}^{\infty}c_l\overline{d_l(\lambda_t)}\\
\sum\limits_{l=1,l\not\in I_t^+}^{\infty}c_ld_l(\lambda_t)&0\end{smallmatrix}\right)\nonumber\\
&~~+\left(\begin{smallmatrix}
0&\sum\limits_{l=1, l\not\in{I_t^-}}^{\infty}c_l\overline{f_l(\lambda_t)}\\
\sum\limits_{l=1, l\not\in{I_t^-}}^{\infty}c_lf_l(\lambda_t)&0\end{smallmatrix}\right)\bigg\},\end{align*} \begin{align*}T_4(k)&=\frac{1}{(k+1)\sin\theta(\lambda_t)}\bigg\{\left(\begin{smallmatrix}
0&\sum\limits_{l\in{I_t^+}}c_l\overline{d(\lambda_t)}\\
\sum\limits_{l\in I_t^+}^{\infty}c_ld_l(\lambda_t)&0\end{smallmatrix}\right)\nonumber\\
&~~+\left(\begin{smallmatrix}
0&\sum\limits_{l\in{I_t^-}}c_l\overline{f_l(\lambda_t)}\\
\sum\limits_{l\in{I_t^-}}c_lf_l(\lambda_t)&0\end{smallmatrix}\right)\bigg\},\end{align*} and $H_k=-\sum\limits_{r=k}^\infty T_3(r)$, with \begin{align*}\left\|H_k\right\|&\leq \frac{T\max\limits_{j\in\{0,\dots,T-1\}}|E_j(\lambda_t)|}{k\sin\theta(\lambda_t)}\left(\sum\limits_{\substack{l=1,\\l\not\in I_t^+}}^\infty \frac{c_l}{\left|\sin\left(\frac{T\omega_t-T\omega_l}{2}\right)\right|} + \sum\limits_{\substack{l=1,\\l\not\in I_t^-}}^\infty \frac{c_l}{\left|\sin\left(\frac{T\omega_t+T\omega_l}{2}\right)\right|}\right)\end{align*} which by \eqref{4.79} and \eqref{4.81} is convergent for any value of the parameter $t$, providing the sequence $(c_j)_{j\in S}$, is chosen so that it satisfies \eqref{4.91} and in which case $\left\|H_k\right\|=O\left(\frac{1}{k}\right)$. Then, since $H_{k+1}-H_k=T_3(k)$, Equation \eqref{4.53} becomes \begin{equation}\label{4.82}\vec{m}_{k+1}=\left(I+T_4(k)+O\left(\frac{1}{k^2}\right)\right)\vec{m}_k.\end{equation}

It is now shown that the sequence $(c_j)_{j\in S}$ can be chosen such that the finitely many resonating terms appearing in $T_4(k)$ do not cancel. This involves the following elementary lemma.

\begin{lemma}\label{4.86}
Let $A\in \mathbb{C}^{n\times n}$ be a matrix where all the diagonal entries are non-zero. Then for any vector $\vec{f}\in\mathbb{C}^{n}$  with all entries positive there exists a vector $\vec{f}'$ arbitrarily close to $\vec{f}$ so that $A\vec{f}'=\vec{v}$ where the entries of $\vec{v}$ are all non-zero.
\end{lemma}

\begin{proof}
Given $\vec{f}$ with all entries positive, if there are no zero entries in $\vec{v}=A\vec{f}$ then the result is already proven. Otherwise, consider the first non-zero entry located at $\vec{v}_{j_1}, j_1\in\{1,\dots,n\}$. Then, alter the vector $\vec{f}$ to become $\vec{f}'$ so that $\vec{f}_i':=\vec{f}_i$ for all $i\neq j_1$ and $\vec{f}'_{j_1}:=\vec{f}_{j_1}+\epsilon$, where $\epsilon>0$ is sufficiently small so that the first $j_1-1$ entries in the new vector $\vec{v}'=A\vec{f}'$ are non-zero. From the fact that the diagonal entries of $A$ are non-zero, the entry $\vec{v}_{j_1}'$ is also non-zero. We repeat the procedure for each subsequent zero entry in the vector $\vec{v}'$, and since the vector is finite-dimensional this completes the proof.~\qedhere
\end{proof}

By Lemma~\ref{4.85} we have that we can partition the sequence $(\lambda_i)_{i\in S}$ into finite disjoint sets. The same is true for the associated sequence $(c_i)_{i\in S}$. Consequently, the $c_i$ which appear in \eqref{4.82} in $T_4(k)$ all belong to the same equivalence class. Then, to establish that there exists a sequence $(c_i)_{i\in S}$ such that $T_4(k)\neq 0$, we partition the $c_i$ elements into their equivalence classes, each of finite size, and to each class apply Lemma~\ref{4.86}: the elements, $c_i$, comprising the entries of the vector $\vec{f}$, whilst each row of the matrix $A$ encodes the exponential and $E(\lambda_i)$ relations that form the sum when resonance occurs. Assuming the positive sequence $(c_l)_{l\in S}$ is chosen to satisfy both \eqref{4.89} and \eqref{4.91} one can use Lemma~\ref{4.86} to vary the sequence, $(c_l)_{l\in S}$ slightly so that resonance appears and the convergence of \eqref{4.89} and \eqref{4.91} remain unchanged. Then, by applying the generalised Janas-Moszynski theorem \cite{17} to eliminate the error term of order $k^{-2}$  we see that the solution of \eqref{4.82} behaves asymptotically like the solution of

\begin{equation}\label{4.56}\vec{m}_{k+1}=\bigg(I+\frac{1}{k\sin\theta(\lambda_t)}\left(\begin{array}{cc}
0&\overline{Y(\lambda_t)}\\
 Y(\lambda_t)&0\end{array}\right)\bigg)\vec{m}_k,\end{equation}
where  $$Y(\lambda_t):=\sum\limits_{l\in I_t^-} c_l\sum\limits_{j=0}^{T-1}E_j(\lambda_t)e^{-i((T-j)\omega_l+\phi_l)}+\sum\limits_{l\in I_t^+}c_l\sum\limits_{j=0}^{T-1}{E}_j(\lambda_t)e^{i((T-j)\omega_l +\phi_l)}\neq 0.$$

 Much like in the individual eigenvalue case, we need to diagonalise the matrix in \eqref{4.56}. The matrix is already Hermitian and has trace zero, and therefore by removing the term $|Y(\lambda_t)|$ as a factor we observe that the new matrix has eigenvalues $1$ and $-1$ and determinant $-1$. Thus, \eqref{4.56} is equivalent to
\begin{align*}
\vec{m}_{k+1}&=W\left[\prod\limits_{r=1}^{k} \bigg(I+\frac{|Y(\lambda_t)|}{r\sin\theta(\lambda_t)}\left(\begin{array}{cc}
1&0\\
0&-1\end{array}\right)\bigg)\right]W^{-1}\vec{m}_1~~~\\
&=W\left(\begin{array}{cc}
\prod\limits_{r=1}^{k}\left(1+\frac{|Y(\lambda_t)|}{r\sin\theta(\lambda_t)}\right)&0\\
0&\prod\limits_{r=1}^{k}\left(1-\frac{|Y(\lambda_t)|}{r\sin\theta(\lambda_t)}\right)\end{array}\right)W^{-1}\vec{m}_1\\
\sim_r \vec{m}_{k+1}&=\left(\begin{array}{cc}
\left(\widetilde{c}_7+o(1)\right) k^{\left|\frac{Y(\lambda_t)}{\sin\theta(\lambda_t)}\right|}&0\\
0&\left(\widetilde{c}_8+o(1)\right) k^{-\left|\frac{Y(\lambda_t)}{\sin\theta(\lambda_t)}\right|}\end{array}\right)\vec{m}_1,
\end{align*} for non-zero constants $\widetilde{c}_7,\widetilde{c}_8$ depending on $\lambda_t\in\sigma_{ell}(J_T)$.

 Returning to the original recurrence relation, this implies that for $\lambda_t$ there exists a decaying (subordinate) solution, $u_k(\lambda_t)\sim k^{-\left|\frac{Y(\lambda_t)}{\sin\theta(\lambda_t)}\right|}$. Since $\lambda_t$ was an arbitrary element of the sequence $(\lambda_i)_{i\in S}$ this concludes the argument.

 \begin{remark}\label{4.3}
 For the case of only finitely many $\lambda_i,i\in\{1,\dots,n\}$, it is possible to amend the proof of Theorem~\ref{transfer3} so that we compute subordinate solutions that reside in the sequence space $l^2(\mathbb{N};\mathbb{C})$. Replacing the potential $(q_k)$ in the theorem by $(c{q}_k)$ replaces $Y(\lambda_t)$ in \eqref{4.56} with $c{Y}(\lambda_t).$  Then, choosing $c$ sufficiently large such that $\left|\frac{cY(\lambda_t)}{\sin\theta(\lambda_t)}\right|>\frac{1}{2}$ for all $t\in\{1,\dots,n\}$ the subordinate solutions will all lie in $l^2$.
 \end{remark}

The previous theorem proves only that the sum of the potentials simultaneously produces subordinate solutions associated to all $(\lambda_i)_{i\in S}$. It has not been shown that a potential of this structure simultaneously satisfies the initial conditions encoded in the periodic Jacobi operator necessary for an eigenvalue to exist for each $\lambda_i,i\in\{1,\dots,n\}$. This, in general, leads to solving a system of non-linear equations. For the case of $n=2$ we explicitly solve this system.

\begin{thm}\label{4.27}
Let $\lambda_1,\lambda_2\in\sigma_{ell}(J_T)$, $\theta(\lambda_1)\neq\frac{\pi}{2}\neq \theta(\lambda_2)$ and assume $E(\lambda_i)\neq 0$ for $i\in\{1,2\}$. Then for $q'_n$ given by \eqref{mult1} for $n\geq 5$, and suitably chosen  $q'_1,q_2',q_3',q_4'$ we have that $\{\lambda_1,\lambda_2\}\subseteq\sigma_{p}(J_T+Q'+R)$ where $Q'$ is a diagonal matrix with entries $q'_n$ of a Wigner-von Neumann structure and order $\frac{1}{n}$ as $n\rightarrow\infty$, and $$R:=\left(\begin{array}{cccc} 0&r&0&\dots\\
r&\ddots\\
0\\
\vdots\end{array}\right),~r\in\mathbb{R}.$$ For a generic choice of $J_T$ we can choose $r=0$.
\end{thm}

\begin{proof} Choose $q_4'$ as given by \eqref{mult1}. By the previous remark there exist non-zero solutions $v_n^{(1)}$, $v_n^{(2)}$ in $l^2$ such that \eqref{4.26} is satisfied. Set $u_n^{(1)}=v_n^{(1)}$, $u_n^{(2)}=v_n^{(2)}$ for $n\geq 3$. Then,  $$a_{n-1}u^{(1)}_{n-1}+a_nu^{(1)}_{n+1}+(q'_n+b_n-\lambda_1)u^{(1)}_n=0,~{\rm for}~n\geq 4$$ and $$a_{n-1}u^{(2)}_{n-1}+a_nu^{(2)}_{n+1}+(q'_n+b_n-\lambda_2)u^{(2)}_n=0~{\rm for}~n\geq 4.$$ Thus, $u_n^{(1)}$, $u_n^{(2)}$ for $n\geq 3$ and $q_n'$ for $n\geq 4$ are now given. However, for $\lambda_1,\lambda_2$ to be embedded eigenvalues, the following system of equations still needs to be satisfied:
\begin{align*}
a_2u^{(1)}_{2}+(q'_3+b_3)u^{(1)}_3+a_3u^{(1)}_4&=\lambda_1u^{(1)}_3\\
a_2u^{(2)}_2+(q'_3+b_3)u^{(2)}_3+a_3u^{(2)}_4&=\lambda_2u^{(2)}_3\\
a_1u^{(1)}_1+(q'_2+b_2)u^{(1)}_2+a_2u^{(1)}_3&=\lambda_1u^{(1)}_2\\
a_1u^{(2)}_1+(q'_2+b_2)u^{(2)}_2+a_2u^{(2)}_3&=\lambda_2u^{(2)}_2\\
(q'_1+b_1)u^{(1)}_1+a_1u^{(1)}_2&=\lambda_1u^{(1)}_1\\
(q'_1+b_1)u^{(2)}_1+a_1u^{(2)}_2&=\lambda_2u^{(2)}_1,
\end{align*}
where $u^{(1)}_1$, $u^{(1)}_2$, $u^{(2)}_1$, $u^{(2)}_2$, $q'_1$, $q'_2$, $q'_3$ are presently undetermined, while $u_3^{(1)}$, $u_3^{(2)}$, $u_4^{(1)}$, $u_4^{(2)}$ are already defined. We consider two cases depending on the values of $u_3^{(1)}$ and $u_3^{(2)}$.

 (Case One) If $u_3^{(i)}\neq 0$, for some $i$, then without loss of generality let $i=1$ and we set $r=0$. Then, we set  $$q'_1=\lambda_1-b_1; q_2'=\frac{(\lambda_2-b_2)u^{(2)}_2-a_2u^{(2)}_3-a_1u^{(2)}_1}{u^{(2)}_2}; q'_3=\frac{(\lambda_1-b_3)u_3^{(1)}-a_3u^{(1)}_4}{u_3^{(1)}}.$$ Choosing $ u_1^{(1)}=-\frac{a_2u_3^{(1)}}{a_1}$; $u^{(1)}_2=0;$ $u^{(2)}_1=\frac{a_1u^{(2)}_2}{\lambda_2-\lambda_1}$; $u^{(2)}_2=\frac{\lambda_2u^{(2)}_3-a_3u^{(2)}_4-(q'_3+b_3)u^{(2)}_3}{a_2}$
 satisfies the six equations listed above, providing the $u_2^{(2)}$ we chose above is non-zero, i.e. $\left(\lambda_2-\lambda_1+\frac{a_3 u_4^{(1)}}{u_3^{(1)}}\right)u_3^{(2)}-a_3u_4^{(2)}\neq0$.

 If $\left(\lambda_2-\lambda_1+\frac{a_3 u_4^{(1)}}{u_3^{(1)}}\right)u_3^{(2)}-a_3u_4^{(2)}=0$ then this implies $u_3^{(2)}\neq 0$. We further subdivide into two cases. If $a_1\leq \frac{1}{2}\left|\lambda_1-\lambda_2\right|$ then set $$x=\frac{\lambda_2-\lambda_1+\sqrt{(\lambda_2-\lambda_1)^2-4a_1^2}}{2}\neq 0$$ and instead choose $$q_1'=\lambda_1-b_1+x; q_2'=-\frac{a_1^2\epsilon}{a_2x}; q'_3=\frac{(\lambda_1-b_3)u_3^{(1)}-a_3u^{(1)}_4}{u_3^{(1)}}+\epsilon$$ and set $u_1^{(1)}=\frac{u_3^{(1)}(-a^2_2+\epsilon(q_2'+b_2-\lambda_1))}{a^2_2}$; $u_2^{(1)}=-\frac{\epsilon u_3^{(1)}}{a_2}$; $u_1^{(2)}=\frac{u_3^{(2)}(-a^2_2+\epsilon(q_2'+b_2-\lambda_2))}{a^2_2};$ $u_2^{(2)}=-\frac{\epsilon u_3^{(2)}}{a_2}$ for any non-zero constant $\epsilon$ chosen such that $u_1^{(1)}\neq 0, u_1^{(2)}\neq 0$. With this choice it is easy to check the six equations listed above are satisfied. If instead $a_1> \frac{1}{2}\left|\lambda_1-\lambda_2\right|$ then we choose $r$ such that $0<a_1+r\leq \frac{1}{2}|\lambda_1-\lambda_2|$ and return to the start of case one.

(Case Two) Here we have $u_3^{(1)}=u_3^{(2)}=0$. Then, since this implies $u_4^{(1)}$ and $u_4^{(2)}$ are both non-zero, it is possible to add an arbitrary perturbation to $q'_4$ so that by continuity $u_3^{(j)}$ becomes non-zero for both $j$. The conditions for case one are then satisfied.~\qedhere
\end{proof}

\begin{remark}
Theorem~\ref{4.27} gives an illustrative example of two embedded eigenvalues constructed by a Wigner-von Neumann type perturbation. There exists another technique that succeeds in embedding infinitely many eigenvalues into the essential spectrum of a period-$T$ Jacobi operator, and we will explore this in an upcoming publication. However, this other technique does not give as explicit a formula for the potential.
\end{remark}

\subsection*{Acknowledgements}
The authors would like to thank G. Teschl for his very useful remarks. E.J. was supported by the Engineering and Physical Sciences Research Council (grant EP/M506540/1). S.N. was supported by the Russian Science Foundation (project no. 15-11-30007), RFBR  grant 16-11-00443a, NCN grant 2013/09/BST/04319 and LMS. S.N. also expresses his gratitude to the University of Kent at Canterbury for the support and hospitality. Finally, we would like to thank the referee for their insightful comments.

\end{document}